\documentclass[10pt,a4paper]{article}  %{amsart} %

\usepackage[draft]{optional}
\usepackage[british,american]{babel}
\usepackage{bbm}
\usepackage{mathrsfs}

\usepackage{amsfonts, amsmath, wasysym}
\usepackage{amssymb,amsthm,
paralist
}
\usepackage{
latexsym,
nicefrac,
}
\usepackage[usenames]{color}

\usepackage{url}

\definecolor{darkgreen}{rgb}{0,0.5,0}
\definecolor{darkred}{rgb}{0.7,0,0}
\usepackage[colorlinks, 
citecolor=darkgreen, linkcolor=darkred
]{hyperref}
\usepackage{esint}
\usepackage{bibgerm}
\usepackage[normalem]{ulem}

\theoremstyle{plain}
\newtheorem{lemma}{Lemma}[section]
\newtheorem{thm}[lemma]{Theorem}
\newtheorem{prop}[lemma]{Proposition}

\theoremstyle{definition}

\newtheorem{rmk}[lemma]{Remark}
\numberwithin{equation}{section}

\newcommand{\m}{\mathcal{M}}

\newcommand{\ck}{\mathcal{K}}
\newcommand{\ci}{\mathcal{I}}
\newcommand{\cl}{\mathcal{L}}

\newcommand{\cu}{\mathcal{U}}

\newcommand{\pl}[2]{{\frac{\partial #1}{\partial #2}}}
\newcommand{\ga}{\gamma}

\newcommand{\de}{\delta}
\newcommand{\om}{\omega}
\newcommand{\Om}{\Omega}

\newcommand{\la}{\lambda}
\newcommand{\La}{\Lambda}

\newcommand{\si}{\sigma}

\newcommand{\Si}{\Sigma}
\renewcommand{\th}{\theta}

\newcommand{\vph}{\varphi}

\newcommand{\ep}{\varepsilon}

\newcommand{\R}{\ensuremath{{\mathbb R}}}
\newcommand{\N}{\ensuremath{{\mathbb N}}}

\newcommand{\downto}{\downarrow}
\newcommand{\upto}{\uparrow}

\newcommand{\lap}{\Delta}

\newcommand{\grad}{\nabla}

\DeclareMathOperator{\inj}{inj}

\newcommand{\norm}[1]{\Vert#1\Vert}  

\def\osc{\mathop{{\mathrm{osc}}}\limits}

\newcommand{\beq}{\begin{equation}}
\newcommand{\eeq}{\end{equation}}
\newcommand{\beqs}{\begin{equation}}
\newcommand{\eeqs}{\end{equation}}
\newcommand{\beqa}{\begin{equation}\begin{aligned}}
\newcommand{\eeqa}{\end{aligned}\end{equation}}
\newcommand{\beqas}{\begin{equation}\begin{aligned}}
\newcommand{\eeqas}{\end{aligned}\end{equation}}
\newcommand{\brmk}{\begin{rmk}}
\newcommand{\ermk}{\end{rmk}}
\newcommand{\partref}[1]{\hbox{(\csname @roman\endcsname{\ref{#1}})}}
\newcommand{\half}{\frac{1}{2}}

\newcommand*\tr{\mathop{\mathrm{tr}}\nolimits}

\newcommand*\supp{\mathop{\mathrm{supp}}\nolimits}
\newcommand*\dist{\mathop{\mathrm{dist}}\nolimits}

\newcommand*\arsinh{\mathop{\mathrm{arsinh}}\nolimits}

\newcommand*\ddt{\frac{d}{dt}}

\newcommand{\pt}{\partial_t}

\newcommand{\abs}[1]{\vert#1\vert} 
 
\newcommand{\eps}{\varepsilon}
\newcommand{\na}{\nabla}

\newcommand{\Hol}{{\mathcal{H}}} %{\Upsilon}

\newcommand{\Col}{\mathcal{C}}

\newcommand{\thin}{\text{-thin}}
\newcommand{\thick}{\text{-thick}}

\newcommand{\halb}{\frac12}

\newcommand{\lan}{\langle}
\newcommand{\ran}{\rangle}
\title{{\sc
Global weak solutions of the Teichm\"uller harmonic map flow
into general targets
}
\\ 
}
\author{ Melanie Rupflin and Peter M. Topping}
\date{\today}
\begin{document}
\maketitle

\begin{abstract}
We analyse finite-time singularities of the Teichm\"uller harmonic map flow -- a natural gradient flow of the harmonic map energy -- and find a canonical way of flowing beyond them in order to construct global solutions in full generality. Moreover, we prove a no-loss-of-topology result at finite time, which completes the proof that this flow decomposes an arbitrary map into a collection of branched minimal immersions connected by curves.
\end{abstract}

\section{Introduction}
The Teichm\"uller harmonic map flow is a gradient flow of the harmonic map energy that evolves a given map $u_0:M\to N$ from a closed oriented surface $M$ of arbitrary genus $\ga\geq 0$ into a closed target manifold $N$ of arbitrary dimension, and simultaneously evolves the domain metric on $M$ within the class of constant curvature metrics. It tries to evolve $u_0$ to a branched minimal immersion -- a critical point of the energy functional in this situation -- but in general there is no such immersion homotopic to $u_0$, so something more complicated must occur.

The development of the theory so far has suggested that the flow should instead decompose $u_0$ into a \emph{collection} of branched minimal immersions from lower genus surfaces. This paper provides the remaining part of the jigsaw in order to prove this in full generality, by analysing the finite-time singularities that may occur, finding a canonical way of flowing beyond them, and analysing their fine structure in order to prove that no topology is lost except by the creation of additional branched minimal immersions and connecting curves. The resulting global generalised solution  will have at most finitely many singular times, together, possibly, with singular behaviour at infinite time that was analysed in \cite{RT,RTZ,HRT}.

Consider the harmonic map energy
$$E(u,g)=\halb\int_M\abs{du}_g^2 \,dv_g$$
acting on a sufficiently regular map $u:M\to (N,g_N)$, and a metric $g$
in the space $\m_c$ of constant (Gauss-)curvature $-1$, $0$ or $1$ 
(depending on the genus) 
metrics on $M$ 
with fixed unit area in the case that the curvature is $0$.
Critical points are weakly conformal harmonic maps $u:(M,g)\to (N,g_N)$, which are then branched minimal immersions \cite{GOR} (allowing constant maps in addition). The gradient flow, introduced in \cite{RT}, can be written with respect to a fixed parameter $\eta>0$ as
\begin{equation}
\label{flow}
\pl{u}{t}=\tau_g(u);\qquad \pl{g}{t}=\frac{\eta^2}{4} Re(P_g(\Phi(u,g))),
\end{equation}
where $\tau_g(u)=\tr_g (\nabla_g du)$ denotes the tension field of $u$, $P_g$ represents the $L^2$-orthogonal projection from the space of quadratic differentials on $(M,g)$ onto the space $\Hol(M,g)$ of \emph{holomorphic} quadratic differentials, and
$\Phi(u,g)$ is the Hopf differential.
The flow decreases the energy $E(t):=E(u(t),g(t))$ according to
\beqa
\label{energy-identity}
\frac{dE}{dt}&=-\int_M 
\left[|\tau_g(u)|^2+\left(\frac{\eta}{4}\right)^2 |Re(P_g(\Phi(u,g)))|^2\right]\\
&=-\|\pt u\|_{L^2}^2-\frac{1}{\eta^2}\|\pt g\|_{L^2}^2\\
&=-\|\tau_g(u)\|_{L^2}^2-\frac{\eta^2}{32}\|P_g(\Phi(u,g))\|_{L^2}^2,
\eeqa
where we use that 
$\|P_g(\Phi(u,g))\|_{L^2}^2=2\|Re(P_g(\Phi(u,g))\|_{L^2}^2$.
We refer to \cite{RT} for further details.

When the genus $\ga$ of $M$ is zero, then there are no nonvanishing holomorphic quadratic differentials, so $g$ remains fixed, and we recover the harmonic map flow \cite{ES}, which has been studied in  detail for two-dimensional domains, cf. \cite{Struwe85}, \cite[Theorem 1.6]{finwind} and the references therein. 
In the case that $\ga=1$, this flow can be shown to be equivalent to a flow of Ding-Li-Liu \cite{Ding-Li-Liu}, as pointed out in \cite{RT}, and analysed in \cite{Ding-Li-Liu} and \cite{HRT}.

\subsection{Construction of a global flow}

%\bcmt{careful wording below to handle nonsmoothness at $t=0$.}

In both cases $\ga=0$ and $\ga=1$, one obtains global weak solutions
starting with any initial  map 
$u_0\in H^1(M,N)$ and any initial metric $g_0\in \m_c$ \cite{Struwe85, Ding-Li-Liu}. 
For $\ga\geq 2$ it was shown in \cite{Rexistence} that a weak solution 
exists on a time interval $[0,T)$, for some $T\in (0,\infty]$, and if $T<\infty$ then the domain must degenerate in the sense that the injectivity radius of $(M,g)$ must approach zero as $t\upto T$.
In all these cases the flow will be smooth away from finitely many times
and as time increases to a singular time the 
map $u$ splits off one or more (but finitely many) nonconstant harmonic 2-spheres, which will then automatically be branched minimal spheres (see e.g. \cite[(10.6)]{EL1} for this latter fact) as bubbling occurs. At each such singular time $\tau$, the 
continuation of this weak solution is constructed by taking a (unique) limit $(u(\tau),g(\tau))\in H^1(M,N)\times  \m_c$ as $t\upto \tau$ 
and continuing the flow past the singular time by restarting with $(u(\tau),g(\tau))$ as new initial data. 
This process gives a unique flow within the class of weak solutions with non-increasing energy.
It was shown in \cite{DT} and \cite[Theorem 1.6]{finwind} that for the harmonic map flow, and in particular for the case $\ga=0$ above, we have no loss of energy and precise control on the bubble scales at these singular times.
A very similar argument establishes the same properties for all genera 
$\ga$, and the case $\ga\geq 2$ even follows directly from Proposition
\ref{prop:conv-nondeg} below that we need for other reasons.
The upshot of this singularity analysis is that the flow map before a singular time can be reconstructed from the flow map after the singular time together with the branched minimal spheres representing the bubbles.

Whenever a \emph{global} weak solution of \eqref{flow} exists, i.e. when $T=\infty$ for $\ga\geq 2$, and in all cases for $\ga=0,1$, 
then it was shown in \cite{RT, RTZ, HRT} (see also \cite{Ding-Li-Liu, Struwe85}) that either the flow subconverges to a branched minimal immersion, or it subconverges to a \emph{collection} of branched minimal immersions.
This collection may consist partly of bubbles, and it may include a limit branched minimal immersion parametrised over the original domain,
but in general, for $\ga\geq 2$, the domain can split into a collection of lower genus closed surfaces, and the map converges to a branched minimal immersion on some or all of these lower genus surfaces.
The way the domain surface can split into lower genus surfaces 
is described by the classical Deligne-Mumford-type description of how hyperbolic surfaces can degenerate, cf. \cite[Theorem A.4]{RT-horizontal}. In particular, when the domain splits, the length of the shortest closed geodesic in the domain will shrink to zero and so-called collar regions around such shrinking geodesics, described by the Collar Lemma of Keen-Randol, see e.g. \cite[Lemma A.1]{RT-horizontal},  will degenerate.
In all cases, if one is careful to capture all bubbles, including those 
disappearing down any degenerating collars, 
it was shown in \cite{HRT} that \emph{all} energy in the limit is accounted for by branched minimal immersions from closed surfaces.
The upshot of this asymptotic analysis is that when a global weak solution exists, for a domain of arbitrary genus, the map $u(t)$ can be reconstructed from the branched minimal immersions we find, connected together with curves. %, for large $t$.
(See \cite{HRT} for precise statements.)

The theory above leaves open the possibility of the flow stopping in finite time in the case $\ga\geq 2$ if it happens that the injectivity radius of the domain converges to zero, i.e. we have collar degeneration as above but in \emph{finite} time.  We showed in \cite{RT3} that the flow exists and is smooth for all time in the case that the target $(N,g_N)$ has nonpositive sectional curvature, mirroring the seminal work of Eells-Sampson \cite{ES} (although the asymptotic behaviour is more elaborate in our situation, with infinite time singularities reflecting the more complicated structure of the space of critical points).
However, in the  case of general targets, the theory above has the major omission that the existence time $T$ for $\ga\geq 2$ could be finite, 
and by such time we cannot expect the flow to have decomposed $u(t)$ into branched minimal immersions.

%%%%%%%%%%%%%%%

In this paper we show how the flow can be continued in a canonical fashion when this domain degeneration occurs, with the continuation being a finite collection of new flows. By repeating this process a finite number of times, we arrive at a global solution that is smooth except at finitely many singular times. Moreover, our analysis of the collar degeneration singularity allows us to account for all `lost topology' at the singular time in terms of branched minimal spheres, some of which may be conventional bubbles, together with connecting curves, despite the tension field diverging to infinity in general. 
Combined with the earlier work described above, a consequence is that the flow realises the following:
\begin{quote}
\em
Any smooth map $u_0:M\to (N,g_N)$ is decomposed by the flow \eqref{flow} into a finite collection of branched minimal immersions $v_i:\Si_i\to (N,g_N)$ from closed Riemann surfaces 
$\{\Si_i\}$ of total genus no more than $\gamma$. The original $M$ can be reconstructed from the surfaces $\{\Si_i\}$ by removing a finite collection of pairs of tiny discs in $\coprod_i \Si_i$ and gluing in cylinders.
The map $u_0$ is homotopic to the corresponding combination of the $\{v_i\}$ together with connecting curves on the glued-in cylinders.
\end{quote}

For other situations in which maps are decomposed into collections of minimal objects, see \cite{MSY} and \cite{HS}, for example.

In order to make a continuation of the flow, we require the following basic description of the convergence of the flow as we approach a finite-time singularity. 
This can be applied to a weak solution (including bubbling) by restricting to a short time interval just prior to a time when the injectivity radius drops to zero, thus avoiding the bubbling and allowing us to consider a smooth flow for simplicity.
A far more refined description will be required later in order to ensure that the continuation after the singularity properly reflects the flow just before.

\begin{thm}
\label{thm:basic_convergence}
Let $M$ be any closed oriented surface of genus $\gamma\geq 2$ and let $(N,g_N)$ be any smooth closed Riemannian manifold. 
Let $(u,g)$ be a smooth solution of \eqref{flow} defined on a time interval $[0,T)$ with $T<\infty$ that is maximal in the sense that
\beq \label{ass:degen} \liminf_{t\upto T} \inj_{g(t)}(M)=0.
\eeq
Then the following properties hold:
\begin{enumerate}
\item
\label{1.1:1}
The `pinching set' $F\subset M$ defined by
\beq 
\label{def:pinching-set} 
F:=\{p\in M:\, \liminf_{t\upto T}\inj_{g(t)}(p)=0\}\eeq
is nonempty and closed, and its complement
$\cu:=M\setminus  F$ is nonempty and supports a complete hyperbolic metric $h$ with finite volume and cusp ends, so that $(\cu, h)$ is conformally equivalent to a finite disjoint union of closed Riemann surfaces $M_i$ with finitely many punctures and genus strictly less than that of $M$, and so that 
$$g(t)\to h \text{ smoothly locally on } \cu \text{ as }t\upto T.$$
\item
\label{1.1:2}
The  `bubble' set 
\beq 
\label{def:S} 
S:=\{p\in \cu\ :\ \exists \ep>0 \text{ s.t. }
\limsup_{t\upto T} E(u(t),g(t),V)\geq \ep \text{ for all  neighbourhoods } V\subset M \text{ of } p\}
\eeq
is a finite set, and
there exists a smooth map $\bar u:\cu\setminus  S\to N$, 
with $\bar u\in H^1(\cu,h,N)$,
such that 
$$u(t)\to \bar u \text{ as }t\upto T$$
smoothly locally in  $\cu\setminus  S$ and weakly locally in $H^1$ on\, 
$\cu$.
Moreover, $\bar u$ extends to a collection of maps $u_i\in H^1(M_i,N)$.
\end{enumerate}
\end{thm}

The convergence of the metric $g(t)$ here should be contrasted with the convergence of a sequence $g(t_n)$, with $t_n\upto T$, that could be deduced from the differential geometric form of Deligne-Mumford compactness (see e.g. \cite[Theorem A.4]{RT-horizontal}). Our convergence
is uniform in time, and does not require modification by diffeomorphisms.

This theorem already tells us enough to be able to define the continuation of the flow beyond time $T$. We simply take each closed Riemann surface $M_i$, equip it with a conformal metric $g_i$ in the corresponding space $\m_c$ of metrics of constant curvature, and restart the flow on each $M_i$ separately with $u_i$ as the initial map. 
The choice of $g_i$ is uniquely determined when the genus of $M_i$ is at least one, but on the sphere it is initially defined only up to pull-back by M\"obius maps. In this case, 
we must find a way of making a canonical choice of $g_i$ in order to obtain
a canonical choice of continuation. We do this by returning to the limit metric $h$, which induces a smooth conformal complete hyperbolic metric of finite area on the sphere with punctures, 
and choose the metric $g_i$ to be the limit $g_\infty$ of the rescaled Ricci flow on the sphere that starts with the metric $h$, as given by the following theorem which follows immediately from a combination of \cite[Theorem 1.2]{revcusp} (see also the simplifications arising from
\cite{TY}) and \cite{Ham88, Cho91} (see also \cite{GT2}).
Note that by Gauss-Bonnet, the volume of the metric $h$ must be $2\pi(n-2)$, where $n$ is the number of punctures.

\begin{thm}
\label{RFcanonical}
Suppose $\{p_1,\ldots,p_n\}\subset S^2$ is a finite set of points and $h$ is a complete conformal hyperbolic metric on $S^2\setminus \{p_1,\ldots,p_n\}$. Then there exists a unique  smooth Ricci flow $g(t)$ on $S^2$, $t\in (0,T)$, $T=\frac{n-2}{4}$, i.e. a smooth complete solution of 
$\pl{g}{t}=-2Kg$ with curvature uniformly bounded below and such that 
$g(t)\to h$ smoothly locally on $S^2\setminus \{p_1,\ldots,p_n\}$ as $t\downto 0$. 
(Here $K$ is the Gauss curvature.)
Moreover, there exists a smooth conformal metric $g_\infty$ on $S^2$ of constant Gauss curvature $1$ such that $\frac{g(t)}{2(T-t)}\to g_\infty$ smoothly as $t\upto T$.
\end{thm}

Theorem  \ref{thm:basic_convergence}, with the aid of Theorem \ref{RFcanonical},
thus establishes that our flow can be continued canonically beyond the singular time $T$ as a finite collection of flows. 
The construction does not require us to stop prior to the singular time $T$ and perform surgery. Instead, we flow right to the singular time, and the surgery we do consists of nothing more than adding points to fill in punctures in the domain (the analogue of adding an arbitrary cap in a traditional surgery argument).

\subsection{No loss of information at finite-time collar degenerations}

At this stage we have however not yet established a very strong connection between the flow prior to a collar degeneration singularity and the flows after the singularity. We need to relate the topology of $M$ to the topology of the surfaces $M_i$, and to relate the topology of the map $u(t)$ prior to the singularity to the flow maps afterwards, and most of this paper will be devoted to achieving this. The former issue is dealt with by the following

\begin{prop}
\label{prop:collar_desc}
In the setting of Theorem \ref{thm:basic_convergence}, the injectivity radius
converges \emph{uniformly} to a continuous limit:
\beq 
\inj_{g(t)}(x)\to\left\{
\begin{aligned}
& \inj_h(x) &\quad &  \text{for }x\in\cu\\
& 0 & \quad &  \text{for }x\in F=M\setminus  \cu, 
\end{aligned}
\right.\label{conv:inj}
\eeq
as $t\upto T$. 
Moreover, the set $F$ from \eqref{def:pinching-set} consists of $k\in\{1,\ldots,3(\ga-1)\}$ components $\{F_j\}$, and 
the total number of punctures in Theorem \ref{thm:basic_convergence} is $2k$.

Furthermore, there exist $\de_0\in(0,\arsinh(1))$ and $t_0\in [0,T)$ 
such that
for every $t\in [t_0,T)$ there are exactly $k$ simple closed geodesics $\si_j(t)\subset (M,g(t))$ with length $\ell_j(t)=L_{g(t)}(\si_j(t))< 2\de_0$ and the lengths of these geodesics decay according to 
\beq \label{est:ell}
\ell_j(t)\leq C (T-t)(E(t)-E(T))\to 0,\text{ as }t\upto T,
\eeq
for some $C=C(\eta,\gamma)$.
In addition, for every $\de\in (0,\de_0]$ and $t\in [t_0,T)$ the set $\de\thin(M,g(t))$ consists of the union of the 
(possibly empty)
disjoint cylindrical `subcollar' regions $\Col_j=\Col_j(t,\de)$ around $\si_j(t)$ which are
isometric to 
\beq
\label{cylinder}
(-X_j, X_j) \times S^1 \text{equipped with the metric }\rho^2_{j}(s)(ds^2+d\theta^2)
\eeq
where 
\beq 
\label{eq:Xj}
X_j=X_j(t,\de)=  \frac{2\pi}{\ell_j(t)}\arccos \left(\frac{\sinh(\frac{\ell_j(t)}{2})}{\sinh \delta}\right), \text{ if }  2\de \geq \ell_j(t), \text{ while } X_j=0 \text{ if } 2\de<\ell_j(t)\eeq
and  
$$\rho_{j}(s)=\rho_{\ell_j(t)}(s)=\frac{\ell_j(t)}{2\pi \cos(\frac{\ell_j(t) s}{2\pi})},$$
and 
for all $t$ sufficiently large (depending in particular on $\de$) we have $F_j\subset \Col_j(t,\de)$.

\end{prop}

The subcollars $\Col_j$ are subsets of collar neighbourhoods of the collapsing simple closed geodesics described by the Collar lemma (see e.g. \cite[Lemma A.1]{RT-horizontal}).
If $\de(t)\downto 0$ 
sufficiently slowly so that
$\de(t)^{-1}(T-t)(E(t)-E(T))\to 0$ as $t\upto T$, then $X_j(t,\de(t))\to \infty$ as $t\upto T$.

This proposition gives us a topological description of how $M$ can be reconstructed from the $M_i$. We remove $2k$ small discs from the $M_i$ at the punctures described in Theorem \ref{thm:basic_convergence}, and glue in cylinders corresponding to the $k$ collar regions from the proposition. (We see that there will be $2k$ punctures.)

The proposition also demonstrates what we must establish in order to relate the flow map before the singularity to the flow maps after the singularity. The continuation of the flow is given in terms of the smooth local limit $\bar u$ on $\cu\setminus  S$ from Theorem \ref{thm:basic_convergence}. Therefore we can potentially lose parts of the map at the points $S$ and parts of the map `at infinity' in $\cu$. As we shall see in part \ref{part1} of Theorem \ref{thm:bubbles}, the loss of energy at points in $S$ 
is entirely accounted for in terms of bubbles, 
i.e. maps $\om_i:S^2\to N$ that are harmonic and non-constant and are thus themselves branched minimal spheres.

On the other hand, we have to be concerned about parts of the map that are lost at infinity in $\cu$. By Proposition \ref{prop:collar_desc}, we must specifically be concerned with the restriction of the flow map $u(t)$ to the collar regions $\Col_j$.
If we view these collar regions conformally as the cylinder \eqref{cylinder} with the flat metric $g_0=ds^2+d\th^2$, then any fixed length portion of an end of these cylinders will have injectivity radius $\inj_{g(t)}$ bounded below by a positive number, uniformly as $t\upto T$, and thus by Proposition \ref{prop:collar_desc}, it will remain in a compact subset of $\cu$ and the map there will be captured in the limit $\bar u$. 
However, this says nothing about what happens away from the ends of the cylinders, and we have to be concerned because the map there need not become harmonic since the tension field is a priori unbounded in $L^2$.
Nevertheless, part \ref{part_neck} of Theorem \ref{thm:bubbles} will show that near enough the centre of these cylinders -- essentially on the $[T-t]^\half\thin$ part of $(M,g(t))$ -- we will be able to describe the map as a collection of bubbles connected together by curves.

This leaves the worry that a little outside this thin region (for example where the injectivity radius is of the order of $[T-t]^{\half-\ep}$) we might accumulate `unstructured' energy that is lost down the collars in the limit, and is not representing any branched minimal immersion or curve, but instead represents some arbitrary map. 
Again, this is ruled out in the following Theorem \ref{thm:bubbles}, part \ref{part:energy_limit}, where we show that all lost energy lives not just on the $[T-t]^{\half}\thin$ part but even on the
$[T-t]\thin$ part.

\begin{thm}
\label{thm:bubbles}
In the setting of 
Theorem \ref{thm:basic_convergence}, we can extract
a finite collection of branched minimal spheres at the singular time in order to obtain no loss of energy/topology in the following sense.
There exists a sequence $t_n\upto T$ such that
\beq 
\label{tau_control}
\left[ \|\tau_{g}(u)(t_n)\|_{L^2(M,g(t_n))}+\|P_{g}(\Phi(u,g))(t_n)\|_{L^2(M,g(t_n))}\right] \cdot(T-t_n)^\half
\to 0,
\eeq
and so that
\begin{enumerate}
\item
\label{part1}
At each $x\in S$, finitely many bubbles (i.e. nonconstant harmonic maps $S^2\to (N,g_N)$) develop as $t_n\upto T$. All of these bubbles develop at scales of order $o\left((T-t_n)^\half\right)$ and they  account for
all of the energy that is lost near $x\in S$, as is made precise in part \ref{part-2-prop} of Proposition \ref{prop:conv-nondeg}.
In particular, if  $\om_1,\ldots,\om_{m}$ is the complete list of bubbles developing at points in $S$ then
\beqa \label{eq:energy-thick}
E_{thick}:=\lim_{\de\downto 0}\lim_{t\upto T} E(u(t),g(t),\de\thick(\cu,h))&=  E(\bar u, h, \cu) + \sum_{l=1}^{m} E(\om_l)\\
&=\sum_{i} E(u_i, M_i)+ \sum_{l=1}^{m} E(\om_l).
\eeqa
\item
\label{part:energy_limit}
All the energy 
$$E_{thin}:=
E(T)
%\lim_{t\upto T}
-E_{thick},$$
$ E(T):=\lim_{t\upto T}E(t)$,
lost down the 
collars concentrates on the $[T-t]\thin$ part in the sense that 
\beq
\label{en_id_thin1}
E_{thin}=\lim_{t\upto T} E(u(t),g(t),[T-t]\thin(M %\setminus F
,g(t))).
\eeq
In fact, we have the more refined information that 
\beq
\label{en_id_thin2}
E_{thin}=\lim_{K\to\infty}\liminf_{t\upto T} E\big(u(t),g(t),[K(T-t)(E(t)-E(T))]\thin(M %\setminus F
,g(t))\big).
\eeq
\item %[(v)] 
\label{part_neck}
The restriction of the maps  $u(t_n)$ to the $(T-t_n)^{\frac 12}\thin$ part of the degenerating subcollars $\Col_j$ from Proposition \ref{prop:collar_desc} has tension
$\|\tau_{g_0}(u(t_n))\|_{L^2}\to 0$ as $n\to\infty$ with respect to $g_0=ds^2+d\th^2$
and hence can be assumed to
converge to a \textit{full bubble branch} as explained in Proposition \ref{prop:bb-convergence} below.
\end{enumerate}
\end{thm}

%\bcmt{Implicit when defining $E_{thick}$ that the limit as $t\upto T$ exists for sufficiently small $\de>0$. It won't exist for larger $\de$, perhaps, if bubbles hang around points with inj rad approx $\de$}

In the following proposition from \cite{HRT}, we write $a_n\ll b_n$, for sequences $a_n$ and $b_n$, if $a_n< b_n$ for each $n$ and $b_n-a_n\to\infty$ as $n\to\infty$.

\begin{prop}[Contents of Theorem 1.9 and Definition 1.10 of  \cite{HRT}]
\label{prop:bb-convergence}
For any sequence of maps $u_n:[-\hat X_n, \hat X_n]\times S^1\to N$, $\hat X_n\to \infty$,
for which the tension with respect to the flat metric $g_0=ds^2+d\th^2$ satisfies $\norm{\tau_{g_0}(u_n)}_{L^2}\to 0$ there exists a subsequence  that \textit{converges to a full bubble branch} in the following sense:

There exist
a finite number of sequences $s_n^m$ (for $m\in \{0,\ldots,\bar m\}$, $\bar m\in \N$)
with $-\hat X_n=: s_n^0\ll s_n^1\ll \cdots\ll s_n^{\bar m}:= \hat X_n$   such that 
\begin{enumerate}
\item
The \emph{connecting cylinders} $(s_n^{m-1}+\la,s_n^m-\la)\times S^1$, $\la$ large, are mapped near curves in the sense that
\beq \label{eq:collars-curves_new} \lim_{\la\to\infty}\limsup_{n\to\infty}\sup_{s\in (s_n^{m-1}+\la,s_n^m-\la)}\osc(u_n;\{s\}\times S^1)=0,
\eeq
for each $m\in \{1,\ldots,\bar m\}$.
\item For each $m\in \{1,\ldots,\bar m-1\}$ (if nonempty) the translated maps
$u_n^m(s,\th):=u_n(s+s_n^m,\th)$  converge  weakly in $H^1$ locally on $(-\infty,\infty)\times S^1$ to a
%(non-trivial) 
harmonic map $\om^m$ 
and strongly in $H_{loc}^1((-\infty,\infty)\times S^1)$ away from a finite number of points %$S^m$
at which bubbles can be extracted in a way that each bubble is counted no more than once,
and so that in this convergence of $u_n^m$ to a bubble branch there is no-loss-of-energy on compact subsets of $(-\infty,\infty)\times S^1$.
Since $(-\infty,\infty)\times S^1$  is
conformally equivalent to the sphere with two points removed, $\om^m$ extends
to a harmonic map from $S^2$. This map can then be considered as a bubble (in particular a 
branched minimal immersion) if it is nonconstant. If it is constant, then there must be a nonzero number of bubbles developing within.
See Theorem 1.5 of \cite{HRT} for details. 
\end{enumerate}
\end{prop}

\brmk
\label{was1.6rmk}
Proposition \ref{prop:conv-nondeg} will give a more general version of part \ref{part1} of Theorem \ref{thm:bubbles}, establishing the no-loss-of-energy property and control on the bubble scales also at finite-time singularities as considered in \cite{Rexistence} at which the metrics do not degenerate. 
As mentioned earlier, the analogue of this result when the underlying surface is $M=S^2$ can already be found in \cite[Theorem 1.6]{finwind} since \eqref{flow} is then just the harmonic map flow. That theorem also elaborates on the sense in which the finite collection of bubbles develop, and the strategy of its proof broadly carries over to our situation here.
\ermk

The key point of Theorem \ref{thm:bubbles} is that the degenerating collars, and indeed the whole surface, can be divided up into two regions: First, the cylinders making up $[T-t]\thin(M,g(t))$ (and even those making up $[T-t]^{1/2}\thin $) are sufficiently collapsed that when we rescale, the evolving map $u$ can be seen to have very small tension and can thus be represented in terms of branched minimal spheres.
Second, on the remaining $[T-t]\thick$ part, the limiting energy is fully accounted for by the energy of the limits $u_i$ and the energy of the bubbles.
This latter assertion is not a priori so clear since one might have a part of the flow map drifting down the collar, always living in a region such as where the injectivity radius is of the order of e.g.
$[T-t_n]^{1/3}$. Such a part of the map would have no reason to look harmonic in any way, and  might carry some nontrivial topology. This `unstructured' energy could in principle drift down the collar not because energy was flowing around the domain, but because the injectivity radius itself is evolving.

The key to ruling out this latter bad behaviour is the following theorem, which gives a more precise description of the convergence of the metric than the one given in Theorem \ref{thm:basic_convergence} and which asserts essentially that by time $t\in [0,T)$, the
metric $g(t)$ has settled down to its limit $h$ on the 
$[T-t]\thick$ part.
As we shall see, this represents an instance of a more general theory from \cite{RT-horizontal} describing the convergence of a general `horizontal curve' of hyperbolic metrics.

\begin{thm}
\label{thm:1}
In the setting of Theorem \ref{thm:basic_convergence}, 
there exists $\bar K<\infty$ depending on $\eta$ and the genus of $M$
(and determined in Lemma \ref{lemma:horiz2}) such that the following holds true: 

The `pinching set' $F\subset M$ defined in \eqref{def:pinching-set} 
can be characterised as 
\beq \label{eq:char-pinching}
F =\bigcap_{t<T} \{p\in M: \inj_{g(t)}(p)< \de_K(t)\},
\eeq
for any $K\geq \bar K$,
where
$\de_K(t):=K(T-t)\left(E(t)-E(T)\right)\downto 0$ as $t\upto T$ and   
$E(T):=\lim_{t\upto T} E(t)$.  
Equivalently, we have
\beq
\label{Uchar}
\cu:=M\setminus F=\bigcup_{t<T}\, [\de_K(t)]\thick(M,g(t)).
\eeq
In addition to the claims on $\cu$, $h$ and the convergence $g(t)\to h$ made in Theorem \ref{thm:basic_convergence},
for any $K\geq \bar K$, $t_0\in [0,T)$ and $t\in [t_0,T)$, we have that for every $l\in \N$
\beq
\label{est1ofthm1}
\|g(t)-h\|_{C^l([\de_K(t_0)]\thick(M,g(t_0)),g(t_0))}+\|g(t)-h\|_{C^l([\de_{K}(t_0)]\thick(\cu,h),h)}
\leq CK^{-\half}
\left[\frac{\de(t)}{\de(t_0)}
\right]^\half,
\eeq
where we abbreviate $\de(t)=\de_1(t)$ and 
where $C$ depends only on $l$, the genus of $M$ and $\eta$.
Furthermore, for $K>0$ sufficiently large (depending on $\eta$, the genus of $M$ and an upper bound $E_0$ for the initial energy) and for all $t_0\in [0,T)$ -- or for arbitrary $K>0$ and $t_0\in [0,T)$  sufficiently large --
we have
\beq\label{est:thm1:3}
\begin{aligned}
\sup_{t\in [t_0,T)}\|g(t)-h\|_{C^l([K(T-t_0)]\thick(M,g(t_0)),g(t_0))}\qquad
\\
\qquad +\sup_{t\in [t_0,T)}\|g(t)-h\|_{C^l([K(T-t_0)]\thick(\cu,h),h)}
&\leq 
C\frac{(E(t_0)-E(T))^\half}{K^\half}
\to 0,
\end{aligned}
\eeq
as $t_0\upto T$, where $C$ depends on $l$, the genus of $M$ and $\eta$. 
\end{thm}

\brmk
Although we do not require it here, one should be able to improve the smooth local convergence $u(t)\to\bar u$ of Theorem \ref{thm:basic_convergence}
to quantitative control on the size of $u(t)-\bar u$ over, say, the 
$[T-t]^\half\thick$ part of $(M,g(t))$, away from $S$, with respect to an appropriate weighted norm. 
\ermk

In summary we obtain that the flow \eqref{flow} decomposes any smooth map $u_0:M\to N$ into a collection of branched minimal immersions $v_i:\Si_i\to N$ through global solutions that are smooth away from finitely many times as follows:
As discussed earlier, at each singular time $t_m$ for which 
$\inj_{g(t)}(M)\nrightarrow 0$ as $t\upto t_m$, 
all of the lost energy is accounted for in terms of bubbles 
$\om_m^j:S^2\to (N,g_N)$, which we add to the collection of minimal immersions  $v_i$ (adding that same number of copies of $S^2$ to the collection of domains $\Si_i$).  
At singular times for which  $\inj_{g(t)}(M)\to 0$ 
the results discussed above apply and we add both the bubbles developing at the singular points $S\subset \cu$ and those that are disappearing down one of the degenerating collars to the set of minimal immersions $v_i$ (again adding the corresponding number of $S^2$'s to the collection of the $\Si_i$'s) and continue the 
flow on 
the closed lower genus surfaces $M_i$ as described above.

If the genus of any of the closed surfaces $M_i$ is $0$ or $1$, then its continuation will be a weak solution that flows forever afterwards according to the theory of the harmonic map flow \cite{Struwe85} or the theory in \cite{Ding-Li-Liu}.
If the genus of any of the surfaces $M_i$ is larger than $1$, then
the subsequent flow might develop a further finite-time singularity at which a collar degenerates, in which case we repeat the process above to continue the flow further still. Each time we restart the flow after a singularity caused by the degeneration of one or more collars, the genus of the surfaces underlying the continued flows will decrease, so repeating the process finitely many times gives us a global weak solution as desired. 
As the energy is conformally invariant, the resulting global solution has non-increasing energy and the total number of singular times $t_m$ is a priori bounded in terms of the genus, the initial energy and the target $(N,g_N)$.

We can relate the domain(s) and map(s) before a singular time $t_m$ to the flow(s) after the singular time as explained above and can thus reconstruct the initial map and the initial domain in terms of the map(s) and domain(s) at any time $t\in (t_m,t_{m+1})$
and the collection of all of the bubbles $v_i$ obtained at the singular times $t_1<..<t_m$ as well as connecting curves on 
cylinders. 

We can then apply the asymptotic analysis as discussed above (principally from \cite{HRT}) to each of the obtained global flows, eventually adding also the bubbles developing at infinite time as well as the limiting maps $u_j^\infty:M^\infty_j \to N$ obtained 
at infinite time,
which are
branched minimal immersions defined on surfaces of total genus no more than $\gamma$, to the collection of the $(\Si_i,v_i)$. 
This gives the decomposition of the initial map into 
branched minimal immersions $v_i:\Si_i\to N$ described earlier on.

This paper is organised as follows. In Section \ref{sect:metric} we carry out the analysis of the metric component of the flow, proving part \ref{1.1:1} of Theorem \ref{thm:basic_convergence} as well as Theorem \ref{thm:1} and Proposition \ref{prop:collar_desc}.
The resulting control on the evolution of the metric then allows us to analyse the map component in the subsequent Section \ref{sect:map}.
In Section \ref{sect:non-deg} we focus on the properties of the map on the non-degenerate part of the surface, stating and proving Proposition \ref{prop:conv-nondeg}, which yields both part \ref{1.1:2} of Theorem \ref{thm:basic_convergence}  as well as 
part \ref{part1} of Theorem \ref{thm:bubbles}. 
Parts \ref{part:energy_limit} and \ref{part_neck} of Theorem \ref{thm:bubbles} are then proven in Section \ref{sect:deg} where we analyse the map component on the degenerating part of the surface.

\emph{Acknowledgements:} 
The second author was supported by 
EPSRC grant number EP/K00865X/1.

\section{Analysis of the metric component}
\label{sect:metric}
In this section we analyse the metric component of the flow, proving first part \ref{1.1:1} of Theorem \ref{thm:basic_convergence}, then 
 Theorem \ref{thm:1}, and finally Proposition \ref{prop:collar_desc}. This analysis is based on the theory of general horizontal curves we developed in \cite{RT-horizontal}, some of which we recall here.

A \textit{horizontal} curve of metrics on a smooth closed oriented surface $M$ of genus at least $2$ is a 
smooth one-parameter family $g(t)$ of hyperbolic metrics on $M$ for $t$ within some interval $I\subset\R$ so that for each $t\in I$, there exists a holomorphic quadratic differential $\Psi(t)$ such that $\pl{g}{t}=Re(\Psi)$. 
This makes $g(t)$ move orthogonally to modifications by diffeomorphisms, as described in 
\cite{RT-horizontal}.

An important property of horizontal curves is that we can bound 
the $C^l$ norm of their velocity, $l\in\N$, in terms of a much weaker norm of $\pt g$ 
and the injectivity radius. In fact, \cite[Lemma 2.6]{RT-horizontal} gives that for any $x\in M$ and $l\in\N$
\beq 
|\pt{g}(t)|_{C^l(g(t))}(x)\leq C[\inj_{g(t)}(x)]^{-\half}\|\pt g (t)\|_{L^2(M,g(t))},
\eeq
$C$ depending only on the genus of $M$ and $l$,
where $|\Om|_{C^l(g)}(x):=\sum_{k=0}^l |\grad_{g}^{k} \Om|_g(x)$,
with $\grad_g$ the Levi-Civita connection, or its extension.

We furthermore recall that
for every point $x\in M$ the map $t\mapsto \inj_{g(t)}(x)$ is locally Lipschitz on the interval $I$ over which $g$ is defined (cf. \cite[Lemma 2.1]{RT-horizontal}) and that 
\beq \label{est:inj-weaker-copy2}
\bigg|\ddt \left[\inj_{g(t)}(x)\right]^{\half}\bigg|\leq K_0
\norm{\partial_t g(t)}_{L^2(M,g(t))}
\eeq 
holds true for a constant $K_0<\infty$ that depends only on the genus of $M$, see \cite[Lemma 2.2]{RT-horizontal}.

These estimates play an important role in the proof of the 
following convergence result for finite length horizontal curves, 
proven in \cite{RT-horizontal}, which we will use to analyse the metric component of the flow. 
In order to state this result, we introduce some more notation:
If $g(t)$ is defined for $t$ in some interval $[0,T)$, then we let 
$$\cl(s):=\int_s^T\|\pt g(t)\|_{L^2(M,g(t))}dt\in [0,\infty]$$
denote the length of the restriction of $g$ to the interval $[s,T)$.
 Given a tensor $\Om$ defined in a neighbourhood of 
some $K\subset M$, we write
\beq
\label{Ck_notation}
\|\Om\|_{C^l(K,g)}:= \sup_K |\Om|_{C^l(g)}.
\eeq

\begin{thm} (\cite[Theorem 1.2]{RT-horizontal})
\label{thm:horizontal}
Let $M$ be a closed oriented surface of genus $\ga\geq 2$,
and suppose $g(t)$ is a smooth horizontal curve in $\m_{-1}$, for $t\in [0,T)$,  
with finite length $\cl(0)<\infty$.
Then there exist a nonempty open subset $\cu\subset M$, whose complement has $k\in\{0,\ldots,3(\gamma-1)\}$ components, 
and
a complete hyperbolic metric $h$ on $\cu$ for which
$(\cu,h)$ is of finite volume and is conformally a finite disjoint union of closed Riemann surfaces 
(of genus strictly less than that of $M$ if $\cu$ is not the whole of M)
 with $2k$ punctures, such that
$$g(t)\to h$$
smoothly locally on $\cu$.
Moreover, defining $\ci:M\to[0,\infty)$ by
$$\ci(x)=\left\{
\begin{aligned}
& \inj_h(x) &\quad &  \text{on }\cu\\
& 0 & \quad &  \text{on }F=M\setminus  \cu,
\end{aligned}
\right.
$$
we have $\inj_{g(t)}\to \ci$ uniformly on $M$ as $t\upto T$, and indeed that
\beq
\label{unif_conv}
\left\|[\inj_{g(t)}]^\half-\ci^\half\right\|_{C^0}\leq K_0\cl(t)\to 0\qquad\text{as }t\upto T,
\eeq
where $K_0$ is chosen as in \eqref{est:inj-weaker-copy2} and depends only on $\ga$. 
Furthermore, for any $l\in\N$ and 
$\de>0$, if we take $t_0\in [0,T)$ sufficiently large so that 
\beq 
\label{ass:t0} 
(2K_0 \cl(t_0))^2< \de, \qquad K_0 \text{ the constant obtained in {\eqref{est:inj-weaker-copy2}} } 
\eeq
then $\de\thick (M,g(s))\subset \cu$ for every $s\in[t_0,T)$, and we have 
for every  $t\in [t_0,T)$
\beq
\label{est:g_minus_h}
\|g(t)-h\|_{C^l(\de\thick(\cu,h),h)}+\|g(t)-h\|_{C^l(\de\thick(M,g(s)),g(s))}\leq C\de^{-\half}\cl(t),
\eeq
where $C$ depends only on $l$ and $\ga$.
\end{thm}

We first apply this result to prove the properties of the metric component claimed in our basic convergence result, i.e. part 
\ref{1.1:1} of Theorem \ref{thm:basic_convergence} %was given as {thm:1}.
To this end we first note that for any smooth solution $(u,g)$ of \eqref{flow} defined on $[0,T)$, $T<\infty$, the metric component is by definition a smooth  horizontal curve. Furthermore, its length is finite as 
\beqa
\label{est:L2-L}
\cl(t)^2 & =\left(\int_t^T\|\pt g(t)\|_{L^2(M,g(t))}dt\right)^2
\leq (T-t) \int_t^T\|\pt g(t)\|_{L^2(M,g(t))}^2dt\\
&\leq \eta^2(T-t) \left(E(t)-E(T)\right),
\eeqa
by \eqref{energy-identity}, where we abbreviate $E(t):=E(u(t),g(t))$ and $E(T):=\lim_{s\upto T}E(s)$.
In particular, defining 
\beq
\label{barKdef}
\bar K:= 5K_0^2 \eta^2,
\eeq
to depend only on $\eta$ and $\ga$,
and defining
\beq\label{def:delta_K} 
\de_K(t):=K (T-t)(E(t)-E(T)),
\eeq
which we will be considering for $K\geq \bar K$ and $t\in [0,T)$,
we have 
\beq
\label{basicK0Lest}
[K_0\cl(t)]^2 \leq \frac15 \de_{\bar K}(t)
\eeq
for all $t\in [0,T)$.
We may thus analyse the metric component $g$ of any solution of \eqref{flow} with the above Theorem \ref{thm:horizontal}. 

In the setting of Theorem \ref{thm:basic_convergence},
the assumption \eqref{ass:degen} that the metric component degenerates as $t$ approaches $T$ combined with the uniform convergence of the injectivity radius furthermore guarantees that the pinching set $F$ must be non-empty.

Part \ref{1.1:1} of Theorem \ref{thm:basic_convergence} concerning the local convergence of $g(t)$ to a limit $h$ and the properties of $h$, $\cu$ and $F$  is thus a direct consequence of Theorem \ref{thm:horizontal} and the fact that $\cl(0)<\infty$.

To prove the refined properties of the metric component stated in Theorem \ref{thm:1} and Proposition \ref{prop:collar_desc}
we shall  use the  following lemma, where $\bar K$ will be chosen as in \eqref{barKdef} above. 

\begin{lemma} \label{lemma:horiz2}
Let $(u,g)$ be a smooth solution of \eqref{flow} on $[0,T)$, $T<\infty$, on a surface $M$ of genus $\ga\geq 2$. %, such that \eqref{ass:degen} holds.
Then there exists a constant $\bar K$ depending only on  $\eta$ and $\ga$
so that the following holds true. 
If we define $\de_K(t)$ as in \eqref{def:delta_K} then 
for every $t_0\in [0,T)$ the assumption \eqref{ass:t0} of Theorem \ref{thm:horizontal} is satisfied for $t_0$ and any $\de>0$ with $\de\geq \de_{\bar K}(t_0)$ and thus estimate
\eqref{est:g_minus_h} holds true for any $t_0\in[0,T)$,
$s,t\in [t_0,T)$,
and any $\de>0$ with $\de\geq \de_{\bar K}(t_0)$.
Furthermore 
\begin{enumerate} 
\item \label{part1_new_lemma} 
For every $K\geq \bar K$ the pinching set $F$ defined in \eqref{def:pinching-set} can be characterised by \eqref{eq:char-pinching}.
\item \label{part2_new_lemma} 
The metrics $(g(t))_{t\in [t_0,T)}$ are uniformly equivalent and their injectivity radii are of comparable size at points $x\in\de_{\bar K}(t_0)\thick (M,g(t_0))$ in the sense that  
for every $s,t\in [t_0,T)$
\beq
\label{est:equiv-metric}
g(s)(x)\leq C_1\cdot g(t)(x) \text{ and }C_1^{-1}\cdot h(x)\leq g(t)(x)\leq C_1 \cdot h(x)
\eeq
and
\beq \label{est:equiv-inj}
\inj_{g(s)}(x)\leq C_2\cdot \inj_{g(t)}(x)
\eeq
where $C_1\geq 1$ depends only on the genus of $M$, 
while $C_2\geq 1 $ is a universal constant.

\item \label{part3_new_lemma} 
For every $K\geq \bar K$, every $x\in \de_K(t_0)\thick (M,g(t_0))$, every $s,t\in [t_0,T)$ and every $l\in \N$ we have 
\beq \label{est:bamberg1} \abs{\pt g(t)}_{C^l(g(s))}(x)\leq C\delta_K(t_0)^{-1/2}\norm{\pt g(t)}_{L^2(M,g(t))}, \eeq
where $C$ depends only on $l$ and the genus of $M$.
\end{enumerate}

\end{lemma}

\begin{proof}[Proof of Lemma \ref{lemma:horiz2}]
We first remark that the claims are trivially true if $\de_K(t_0)=0$ and hence 
$g\vert_{[t_0,T)}$ is constant in time, so we may assume without loss of generality that $\de_K(t_0)>0$.

Define $\bar K$ as in \eqref{barKdef}. Then \eqref{basicK0Lest}
implies that
\beq
\label{check_hyp}
(2K_0\cl(t_0))^2\leq \frac45 \de_{\bar K}(t_0)<\de_{\bar K}(t_0),
\eeq
and so \eqref{ass:t0} is satisfied for 
%$t_0$ and 
$\de\geq \de_{\bar K}(t_0)$ as claimed in the lemma.

To prove part \ref{part1_new_lemma} of the lemma we combine 
\eqref{unif_conv}  with 
%\eqref{check_hyp} 
\eqref{basicK0Lest} 
to obtain that  
$\inj_{g(t)}(p)\leq (K_0 \cl(t))^2 < \de_{\bar K}(t)$ for every $p\in F$ and every $t\in[0,T)$ and thus that
$$F \subset \bigcap_{t\in [0,T)} \de_K(t)\thin (M,g(t)) \text{
for any } K\geq \bar K.$$
As the reverse inclusion is trivially satisfied this establishes the characterisation \eqref{eq:char-pinching} of the pinching set for each $K\geq \bar  K$.

The proofs of parts \ref{part2_new_lemma} and \ref{part3_new_lemma} of the lemma
are now based on 
estimates on the velocity and the injectivity radius that were derived in \cite{RT-horizontal} for general horizontal curves under the same hypothesis that \eqref{ass:t0} holds true:
 Lemma 3.2 and Remark 3.5 of \cite{RT-horizontal} establish that 
 \eqref{est:equiv-metric} and \eqref{est:equiv-inj} 
hold true for arbitrary horizontal curves, times $s,t\in [t_0,T)$ and points $x\in \de\thick(M,g(t_0))$ provided $t_0$ and $\de$ are so that \eqref{ass:t0} is satisfied. Combined with \eqref{check_hyp} this immediately yields part \ref{part2_new_lemma} of the 
lemma. 
Finally, \eqref{est:bamberg1}, and hence part \ref{part3_new_lemma}
of the lemma follows immediately from
\cite[Lemma 3.2]{RT-horizontal}, with $\de$ there equal to 
$\de_{K}(t_0)$ here, because the hypotheses of that lemma are implied by \eqref{check_hyp}.
\end{proof}

Parts \ref{part2_new_lemma} and \ref{part3_new_lemma} of Lemma \ref{lemma:horiz2} will be used in the next section for the fine analysis of the map component,  but before that we complete the proofs of Theorem \ref{thm:1} and Proposition \ref{prop:collar_desc}.

\begin{proof}[Proof of Theorem \ref{thm:1}]
We let $\bar K$ be the constant obtained in Lemma \ref{lemma:horiz2},
i.e. given by \eqref{barKdef}, and set as usual $\de_K(t)=K(T-t)(E(t)-E(T))$.
For this choice of $\bar K$ the characterisation \eqref{eq:char-pinching} of the pinching set $F$ has already been proven in Lemma \ref{lemma:horiz2} and from this lemma we furthermore know that \eqref{ass:t0} holds true for any $t_0$ and any $\de\geq \de_{\bar K}(t_0)$ and thus in particular for $\de=\de_K(t_0)$, $K\geq \bar K$. Hence \eqref{est1ofthm1} follows from the corresponding estimate \eqref{est:g_minus_h} of Theorem \ref{thm:horizontal}
and the bound \eqref{est:L2-L} on $\cl(t)$.

It remains to prove \eqref{est:thm1:3}. For this we observe that 
 for $K>0$ sufficiently large 
and for all $t_0\in [0,T)$ -- or for arbitrary $K>0$ and $t_0\in [0,T)$  sufficiently large -- we can 
be sure that $\bar K(E(t_0)-E(T))\leq K$ and hence 
by \eqref{basicK0Lest} that 
\eqref{ass:t0} is satisfied for $t_0$ and $\de=K(T-t_0)$. 
This allows us to 
apply estimate \eqref{est:g_minus_h} of Theorem \ref{thm:horizontal} also for such values of $\de$ 
which then gives that 
\beqa
\label{h_est_lem_app2}
\sup_{t\in [t_0,T)}\|g(t)-h\|_{C^l([K(T-t_0)]\thick(M,g(t_0)),g(t_0))}
&\leq 
C\frac{\cl(t_0)}{K^\half(T-t_0)^\half}\\
&\leq 
C\frac{(E(t_0)-E(T))^\half}{K^\half}
\to 0,
\eeqa
as $t_0\upto T$, using \eqref{est:L2-L}, as well as that 
\beq
\sup_{t\in [t_0,T)}\|g(t)-h\|_{C^l([K(T-t_0)]\thick(\cu,h),h)}\leq 
C\frac{(E(t_0)-E(T))^\half}{K^\half}
\to 0,
\eeq
where $C$ depends only on $l$, $\eta$ and the genus of $M$.
This completes the proof of Theorem \ref{thm:1}.

\end{proof}

\begin{proof}[Proof of Proposition \ref{prop:collar_desc}]
The uniform convergence of the injectivity radius  follows from Theorem \ref{thm:horizontal} as  $(g(t))_{t\in [0,T)}, \, T<\infty$, is a horizontal curve of finite length.

We furthermore recall  that standard results from the theory of hyperbolic surfaces give that for any $\de<\arsinh(1)$ the $\de\thin$ part of a hyperbolic surface is given by the union of disjoint subcollar regions around the simple closed geodesics of length $\ell<2\de$ as
described in the proposition and refer to the appendix of \cite{RT3} as well as the references therein  for further details.

For $K\geq K_0$, $K_0$ as in \eqref{est:inj-weaker-copy2},
we define closed sets
$$F_K(t):=\{p: \inj_{g(t)}(p)\leq (K\cl(t))^2\},$$ 
for $t\in [0,T)$.
It follows from the slightly stronger result \cite[Lemma 3.1]{RT-horizontal} that 
the sets $F_K(t)$ are nested, becoming only smaller as $t$ increases, and that 
the pinching set $F$ can be written as
\beq
\label{anotherFchar}
F=\bigcap_{t\in [0,T)}F_K(t).
\eeq
It is useful for us to appeal to this fact for some $K>K_0$, and we choose $K=2K_0$.

Thus for $t_0$ sufficiently large, chosen in particular so that $(2K_0\cl(t_0))^2<\arsinh(1)$, the pinching set 
$F$ has  the same number $k\in \{1,...,3(\gamma-1)\}$  of connected components as the
sets $F_{2K_0}(t) $, $t\in [t_0,T)$, with the connected components of $F_{2K_0}(t) $ being disjoint closed subcollars around geodesics $\si_j(t)$ of length $\ell_j(t)\leq 2(2K_0\cl(t))^2\leq C(T-t)(E(t)-E(T))$ whose interior is as described in the proposition.
In particular given any $\de \in (0,\arsinh(1))$ and  $t\in [t_0,T)$ sufficiently large (depending in particular on $\de$), 
we know that the connected components $F_j$ of the pinching set are contained in the corresponding subcollar $\Col_j(t,\de)$ as claimed in the proposition.

It thus remains to show that there exists a number $\de_0\in (0,\arsinh(1))$ so that any simple closed geodesic in $(M,g(t))$, $t\in [t_0,T)$, that does not coincide with one of the $\si_j(t)$ 
%(up to change of orientation) 
must have length at least $2\de_0$. To this end we observe  that
the characterisation \eqref{anotherFchar} this time with $K=K_0$ gives
%as $F=\bigcap_{t<T} F_{K_0}(t)$ we know that the set 
$$\Om:=(2K_0\cl(t_0))^2\thick (M,g(t_0))\subset
M\setminus F_{K_0}(t_0)\subset \cu$$ 
and since $\Om$ is closed, it is a compact subset of $M$ and hence also of\, $\cu$.
Therefore over $\Om$ the injectivity radius $\inj_{g(t)}(\cdot)$, $t\in[t_0,T)$, is bounded uniformly from below by some constant $\de_0\in (0,\arsinh(1))$ thanks to \eqref{conv:inj}. 
Consequently, any simple closed geodesic in $(M,g(t))$ that enters $\Om$ must have length at least $2\de_0$.

The only alternative is that the simple closed geodesic 
in $(M,g(t))$ is  fully contained in one of the $k$ cylinders $\Col_j(t_0,(2K_0\cl(t_0))^2) $ 
in which case it must be homotopic to $\si_j(t)$ (up to change of orientation) and hence coincide with $\si_j(t)$.
\end{proof}

\section{Analysis of the map component}
\label{sect:map}

The challenges of analysing the map component are of a different nature depending on whether we consider a region where the metric has already settled down or a region in a collar that will ultimately degenerate.
Roughly speaking, on the non-degenerate part of the surface we control the metric but cannot hope to bound the tension while on the degenerating part of the surface the metric is not controlled but the tension tends to zero when computed with respect to the flat metric in collar coordinates along a sequence of times $t_n\upto T$ as considered in Theorem \ref{thm:bubbles}.
We will analyse the map component separately on these two different regions, with the analysis on the non-degenerate part, and hence the proofs of part \ref{1.1:2} of Theorem \ref{thm:basic_convergence} and of
part \ref{part1} of Theorem \ref{thm:bubbles}, carried out in Section \ref{sect:non-deg}. Parts \ref{part:energy_limit} and \ref{part_neck} of Theorem \ref{thm:bubbles}, which concern the part of the map that is lost on degenerating collars, are then proven in Section \ref{sect:deg}. In both of these sections we use a
local energy estimate that is derived in Section \ref{sect:energy-est}

\subsection{Local energy estimates}
\label{sect:energy-est}

The goal of this section is to prove the following lemma on the evolution of cut-off energies 
 \beq \label{def:cut-energy} 
 E_\vph(t):=\frac12\int\vph^2\abs{du(t)}_{g(t)}^2 dv_{g(t)},
 \eeq
for functions $\varphi\in C^\infty(M,[0,1])$.

\begin{lemma}[Local energy estimate] 
\label{lemma:cut-energy}
Let  $(u,g)$ be a smooth solution of \eqref{flow} on a closed surface of genus at least $2$, and for an interval $[0,T)$, $T<\infty$, and let 
 $\varphi\in C^\infty(M,[0,1])$ be such that there exists $t_0\in [0,T)$  and $K\geq \bar K$, for $\bar K$ the constant obtained in Lemma \ref{lemma:horiz2}, so that 
\beq \label{ass:cut-off} \supp(\vph)\subset \delta_K(t_0)\thick (M,g(t_0)),\eeq 
where as usual $\de_K(t):= K(T-t)(E(t)-E(T)).$

Then the limit $\lim_{t\upto T} E_\vph(t)$ exists and (assuming the flow is not constant in time on $[t_0,T)$)  for any  $t_0\leq t< s<T$ we have
\beqa 
\label{est:Ephi-new}
\abs{E_\vph(t)-E_\vph(s)}
&\leq  E(t)-E(s)+C\big[\de_K(t_0)^{-\frac12}+\norm{d\vph}_{L^\infty(M,g(t_0))}
\big] (s-t)^{\half} (E(t)-E(s))^\half\\
&\leq  E(t)-E(T)+C\big[\de_K(t_0)^{-\frac12}+\norm{d\vph}_{L^\infty(M,g(t_0))}
\big] (T-t)^{\half} (E(t)-E(T))^\half
\eeqa
where $C$ depends only on the coupling constant $\eta$, the genus of $M$ and an upper bound $E_0$ for the initial energy.
\end{lemma}

A first step in the proof of Lemma \ref{lemma:cut-energy} is to show the following analogue of well-known local energy estimates for harmonic map flow as found e.g. in  \cite[Section 2]{finwind}.

\begin{lemma}
\label{lemma:loc-energy}
 Let $(u,g)$ be a (smooth) solution of \eqref{flow} on $[0,T)$ and let $\varphi\in C^\infty(M,[0,1])$ be
 any given function.
 Then the evolution of the cut-off energy $E_\vph(t)$ defined in \eqref{def:cut-energy} 
 is controlled by 
 \beqa \label{est:Ephi-1}
  \bigg|\tfrac{d}{dt} E_\vph +\int \vph^2\abs{\tau_g(u)}^2 dv_g\bigg|
  &\leq %\tfrac12 \int \vph^2\abs{\tau_g(u)}^2 dv_g  +
  2\sqrt{2} E(u,g)^{1/2} \norm{d\vph}_{L^\infty(M,g)}
   \big(\int \vph^2\abs{\tau_g(u)}^2 dv_g\big)^{1/2}\\
  &\qquad + \norm{\pt g}_{L^\infty(\text{supp}(\vph),g)} E_\vph
  .
  \eeqa

\end{lemma}

\begin{proof}
The equation of the map component can be described by 
 \beq\label{eq:evol-u} 
 \pt u-\Delta_gu=A_g(u)( du, du)=g^{ij}A(u)(\partial_{x_i} u,\partial_{x_j} u)\perp T_uN\eeq
 if we view 
 $(N,g_N)$ as a submanifold of Euclidean space using Nash's embedding theorem and denote by $A$  the second fundamental form of $N\hookrightarrow \R^K$. We multiply this equation
 with $\vph^2 \pt u$ and integrate over $(M,g)$ to obtain
  $$
 0 %=\int\vph^2\abs{\pt u}^2 dv_g-\int \vph^2 \Delta_g u\pt udv_g\\
 =\int\vph^2\abs{\pt u}^2 dv_g+\int \langle \pt d u, d u\rangle_g \vph^2 dv_g+\pt u \langle du,d(\vph^2)\rangle_g dv_g
 .$$
 We now recall that $\pt g$ is given as the real part of a quadratic differential and thus has zero trace, which implies that
$\frac{d}{dt}dv_g=0$. As $\vph$ is independent of time while $\pt u=\tau_g(u)$ we thus obtain
\beqa 
\bigg|\tfrac{d}{dt}  E_\vph+\int\vph^2\abs{\tau_g(u)}^2 dv_g\bigg|&\leq 
\tfrac12\tfrac{d}{d\eps}\vert_{\eps=0}\int \abs{du}_{g(t+\eps)}^2\vph^2 dv_g  + 2\int \vph\abs{d\vph}\cdot \abs{\tau_g(u)}\cdot \abs{du} dv_g 
\\
  &\leq \norm{\pt g}_{L^\infty(\supp(\vph),g)}E_\vph \\
  & \quad+ 2 \norm{d\vph}_{L^\infty(M,g)}(2E(u,g))^{1/2}\cdot  \bigg(\int \vph^2\abs{\tau_g(u)}^2 dv_g\bigg)^{1/2}
  \eeqa
as claimed.
\end{proof}

Based on this lemma as well as the control on the metric on the $\de_K(t_0)\thick$ part of the domain obtained in Lemma \ref{lemma:horiz2}, we can now prove our main energy estimate.

\begin{proof}[Proof of Lemma \ref{lemma:cut-energy}]
Given $K\geq \bar K$, with $\bar K$ as in Lemma \ref{lemma:horiz2}, we set as usual $\de_K(t):=K(T-t)(E(t)-E(T))$ and consider a cut-off function $\vph$ as in the lemma for which  \eqref{ass:cut-off} is satisfied for some $t_0$. 

This assumption on the support of $\vph$ allows us to bound any 
$C^l$ norm of $\pt g$ on $\supp(\vph)$ using estimate \eqref{est:bamberg1} of Lemma \ref{lemma:horiz2}, which implies in particular that  
\beq 
\norm{\pt{g}(t)}_{L^\infty(\supp \varphi,g(t))} \leq C[\de
_K(t_0)]^{-\half}\|\pt g(t)\|_{L^2(M,g(t))}, \text{ for any } t\in [t_0,T)
\eeq
holds true with a constant $C$ that depends only on the genus. 

Furthermore, the equivalence \eqref{est:equiv-metric} of the metrics on $\de_{\bar K}(t_0)\thick(M,g(t_0))$, and thus in particular on $\supp(\vph)$, obtained in the same lemma allows us to bound
 $$\norm{d\vph}_{L^\infty(M,g(t))}\leq \sqrt{C_1} \norm{d\vph}_{L^\infty(M,g(t_0))} \text{ for  } t\in [t_0,T).$$
The local energy estimate \eqref{est:Ephi-1} of Lemma \ref{lemma:loc-energy} thus reduces to
\beqa  
\label{est:Ephi_proof}
  \abs{\tfrac{d}{dt} E_\vph}
  \leq & \norm{\tau_g(u)}_{L^2(M,g)}^2
  +C\norm{d\vph}_{L^\infty(M,g(t_0))}\norm{\tau_g(u)}_{L^2(M,g)} 
  +C\delta_K(t_0)^{-\half} \norm{\pt g}_{L^2(M,g)} \\
  \leq &\big( -\tfrac{dE}{dt}\big) +C\cdot \big[\norm{d\vph}_{L^\infty(M,g(t_0))}+\delta_K(t_0)^{-\half}\big] \big(-\tfrac{dE}{dt}\big)^\half
\eeqa
for $t\in [t_0,T)$,
where the constant $C$ now depends not only on the genus but also on the coupling constant and an upper bound $E_0$ on $E(0)\geq E(t)$ and where we used the evolution equation \eqref{energy-identity} of the total energy in the second step. 

  Integrating \eqref{est:Ephi_proof} over $[t,s]\subset [t_0,T)$ yields the claim of Lemma \ref{lemma:cut-energy}. 
\end{proof}

Lemma \ref{lemma:cut-energy} will allow us to determine both the part of the degenerating collars where energy can be lost as well as the scale at which energy concentrates around points in the singular set $S\subset \cu$. This will then allow us to capture two collections of bubbles, one 
developing at the bubble points $S$, but also an additional collection that are disappearing down the collars. By further analysing what can happen between these bubbles, this will allow us to prove Theorem \ref{thm:bubbles}.

This bubbling analysis will be carried out along a sequence of times $t_n$ for which \eqref{tau_control} holds. The existence of such $t_n\upto T$ follows by a standard argument: Integrating \eqref{energy-identity} in time
%and using \eqref{normalisation},
implies that 
$$\int_0^T \|\tau_g(u)\|_{L^2}^2 dt
\qquad\text{ and }\qquad
\int_0^T \|P_g(\Phi)\|_{L^2}^2 dt$$
are bounded in terms of an upper bound $E_0$ for the initial energy, and (for the second integral) the coupling constant $\eta$.
Here we suppress the dependence of $\Phi$ on $u$ and $g$, of the $L^2$ measure on $g$, and of $u$ and $g$ on $t$.
These bounds imply that whenever a smooth function $f:[0,T)\to [0,\infty)$ has infinite integral, there exists a sequence of times $t_n\upto T$ such that 
$$\left[\|\tau_g(u)\|_{L^2}^2+\|P_g(\Phi)\|_{L^2}^2\right](t_n)<f(t_n).$$
In particular, we may always choose some sequence $t_n\upto T$ so that 
\eqref{tau_control} holds true for $t_n$ (and thus also for any subsequence that we take later).

\subsection{Analysis of the map component on the non-degenerate part of the surface}
\label{sect:non-deg}
On compact subsets of $\cu$ we can control the metric component using Lemma \ref{lemma:horiz2} and may thus think of the evolution equation \eqref{flow} for the map component as a solution of a flow that is akin to the classical harmonic map flow albeit with a (well controlled) time dependent metric. This will allow us to adapt well-known techniques from the theory of the harmonic map flow, in particular from \cite{Struwe85} and \cite{finwind}, to analyse the solution on this part of the domain in detail: We 
 prove that as $t\upto T$ energy concentrates only at finitely many points $S$ away from which the maps converge in $C^l$, for each $l\in\N$, and that, along a subsequence of times $t_n\upto T$ as in \eqref{tau_control}, we can extract a finite number of bubbles at each point in $S$ which account for all the energy that is lost near these point. This last part is equivalent to proving that no energy is lost on so-called neck-regions around the bubbles (not to be confused with collar regions around the degenerating geodesics). This fine analysis of the map component on the thick part of the surface applies not only in the case of a finite time degeneration as considered in the present paper but (as a by-product of the following proposition) also gives refined information at singular times as considered in \cite{Rexistence} across which the metric remains controlled.

\begin{prop}[cf. \cite{finwind}]
\label{prop:conv-nondeg}
Let $(u,g)$ be any smooth solution of \eqref{flow} for $t\in[0,T)$ on a surface of genus at least $2$. Let $F$ be the (possibly empty) set given by \eqref{def:pinching-set} and let $S$ be defined as in 
%Theorem \ref{thm:basic_convergence}. 
\eqref{def:S}.
Then $S$ is a finite set and 
 \begin{enumerate}
  \item \label{part-1-prop}$u(t)\text{ converges smoothly locally on } M\setminus (F\cup S)$ and weakly locally in $H^1$ on $M\setminus F$ as $t\upto T$,
to a limit that we denote $u(T)$.
  \item \label{part-2-prop} 
  We have no loss of energy at points in $S$, and the scales of bubbles developing at the points of $ S$ (along a subsequence of times $t_n\upto T$ as in \eqref{tau_control}) are small compared with 
$(T-t_n)^\half$. Indeed, if $\om_1,\ldots,\om_{m'}$ are the bubbles developing at $x\in S$ then for every $\nu>0$
\beqa \label{eq:energy-id-balls}\lim_{r\downto 0}\lim_{t\upto T}E(u(t),g(t),B_{h}(x,r))&= \lim_{r\downto 0}\lim_{t\upto T}E(u(t),g(t),B_{g(t)}(x,r))
\\
&=
\lim_{t\upto T}E(u(t),g(t),B_{g(t)}(x,\nu(T-t)^\half))
=\sum_{l=1}^{m'} E(\om_l).\eeqa
In particular, if $\om_1,\ldots,\om_{m''}$ is the complete list of bubbles developing at points in $S$ and if $\Om\subset\subset \cu$ 
is chosen large enough so that $S$ is contained in the interior of $\Om$ then 
\beqa \label{eq:energy_id_compact}
\lim_{t\upto T} E(u(t),g(t),\Om)= E(\bar u, h, \Om) + \sum_{l=1}^{m''} E(\om_l).
\eeqa
   \end{enumerate}
\end{prop}
In the setting of Theorem \ref{thm:basic_convergence}, i.e. in case that $\inj_{g(t)}(M)\to 0$ as $t\upto T$,  part \ref{part-1-prop} of the proposition 
 yields the  convergence of the maps $u(t)$ on $\cu$ respectively on $\cu \setminus S$ claimed in part \ref{1.1:2} of Theorem  \ref{thm:basic_convergence}. As the resulting limiting maps can be extended across the punctures to $H^1$ maps from $M_i$ (since their energy is bounded) and as the properties of the metric component claimed in part \ref{1.1:1} of Theorem \ref{thm:basic_convergence} have already been proven in Section \ref{sect:metric}, 
this then completes the proof of Theorem \ref{thm:basic_convergence},
modulo the proof of Proposition \ref{prop:conv-nondeg}.  

The second part of Proposition \ref{prop:conv-nondeg} implies part \ref{part1} of Theorem \ref{thm:bubbles}: The first part of \eqref{eq:energy-thick} follows from \eqref{eq:energy_id_compact} since $\de\thick(\cu,h)$ is compact for every $\de>0$, 
while the second part of \eqref{eq:energy-thick} is due to the conformal invariance of the energy.

For the proof of Proposition \ref{prop:conv-nondeg} we shall use the following standard 
 $\eps$-regularity result.

%
% cf. Huxol thesis section 4.1.2
\begin{prop} %(see e.g. \cite{DT}) %BUT THAT PROOF WON'T WORK since it uses elliptic regularity
\label{prop:eps-reg}
 There exist constants $\eps_0>0$ and $C\in \R$ depending only on the target manifold so that the following holds true.
 Let $u:B_{g_H}(x,r)\to N$ be any smooth map from a ball of radius $r\in (0,1]$ in the hyperbolic plane $(H,g_{H})$ with energy
 $$E(u,g_H, B_{g_H}(x,r))\leq\eps_0.$$
 Then 
 \beq 
\label{est:H^2-estimate}
\int\vph^2\big[\abs{\na_{g_H} du}_{g_H}^2+\abs{du}_{g_H}^4 \big]\,dv_{g_H}\leq C\norm{d\vph}_{L^\infty(H,g_H)}^2E(u,g_H, B_{g_H}(x,r))+C\int\vph^2 \abs{\tau_{g_H}(u)}^2 dv_{g_H}
\eeq
holds true for every function $\vph\in C_c^\infty(B_{g_H}(x,r),[0,1])$.
\end{prop}

%%see Melanie email of 4 Sept 2017 for proof, for example
Note that the Hessian term $\abs{\na_{g_H} du}_{g_H}^2$ is not referring to the intrinsic Hessian. That term is instead the sum of the corresponding terms for each component of $u$ viewed as a map into Euclidean space, and depends on the isometric embedding of $N$ that we chose.
This term can be controlled in terms of the integral of $\vph^2|\lap_g u|^2$ and lower order terms simply using integration by parts. This leading order term can be rewritten using 
\eqref{eq:evol-u} and the resulting quartic term in $du$ controlled with the Sobolev inequality.
The details of a very similar argument can be found in \cite[Proposition 2.4]{melanie_uniqueness}.

\begin{proof}[Proof of Proposition \ref{prop:conv-nondeg}]
Part \ref{part-1-prop} of the proposition represents the analogue of Lemma 3.10' of \cite{Struwe85}  and we shall use properties of horizontal curves from Lemma \ref{lemma:horiz2} to 
 control the evolution of the metric, see also \cite{Rexistence} for a related proof in the non-degenerate case.

In the following we shall use several times that
for any compact subset $\Om \subset \cu=M\setminus F$ there exists a number 
$t_0=t_0(\Om)\in [0,T)$ so that 
\beq \label{ass:compact} \Om\subset \de_{2\bar K}(t_0)\thick (M,g(t_0)) \eeq
  where $\de_{K}(t)=K(T-t)(E(t)-E(T))$ and $\bar K$ are as in Lemma \ref{lemma:horiz2}.
  Indeed, for solutions of \eqref{flow} which degenerate as described in  \eqref{ass:degen}, this is a consequence of 
the uniform convergence of the injectivity radius obtained in Proposition \ref{prop:collar_desc}, while otherwise $\inj_{g(t)}(M)$ is bounded away from zero uniformly so \eqref{ass:compact} is trivially satisfied for $t_0$ sufficiently close to $T$. 
 As a consequence of \eqref{ass:compact} also 
  \beq \label{est:inj-Om}
\inj_{g(t_0)}(x)\geq \de_{\bar K}(t_0) \text { for all } x\in M  \text{ with } \dist_{g(t_0)}(x,\Om)\leq \de_{\bar K}(t_0),
\eeq 
which allows us to apply Lemma \ref{lemma:horiz2} to control the evolution of the metric as well as Lemma \ref{lemma:cut-energy}  to bound the cut-energy on this neighbourhood of $\Om$.

We first apply this idea to prove that 
for any point $p\in \cu$ for which 
\beq \label{def:tilde-S}
\limsup_{t\upto T} E(u(t),g(t),V)\geq \ep_0 \text{ for every neighbourhood }
V\subset M \text{ of } p,\eeq 
$\eps_0>0$ the constant obtained in Proposition \ref{prop:eps-reg}, 
we also  have
\beq 
\label{est:tildeS}\liminf_{t\upto T} E(u(t),g(t),W)\geq \ep_0
\quad\text{ for every neighbourhood }
W\subset M \text{ of } p.\eeq 
 In particular, the set $\tilde S$ of points in $\cu$ for which \eqref{def:tilde-S} holds is a finite set and we will later see that it agrees with the singular set $S$ defined in \eqref{def:S}. %and is thus in particular finite.
  
To show \eqref{est:tildeS} for a given $p\in \tilde S$ we let $t_0\in [0,T)$ be large enough so that \eqref{ass:compact} holds for $\Om=\{p\}$. Given any 
 neighbourhood $W$ of $p$ we then choose $r\in (0,\de_{\bar K}(t_0))$ small enough so that $B_{g(t_0)}(p,r)\subset W$ 
 and select a cut-off function $\vph\in C_c^\infty(B_{g(t_0)}(p,r),[0,1])$ with $\vph\equiv 1$ in a neighbourhood $V$ of $p$.

Lemma \ref{lemma:cut-energy} implies that the 
limit $\lim_{t\upto T}E_\vph(t)$ of the cut-off energy defined in \eqref{def:cut-energy} exists and thus that, by \eqref{def:tilde-S}, 
$$\liminf_{t\upto T} E(u(t),g(t),W)\geq \lim_{t\upto T}E_\vph(t)\geq \limsup_{t\upto T} E(u(t),g(t),V)\geq \eps_0$$ as claimed. 
Having thus established that there is only a finite subset $\tilde S$ of points in $\cu$ for which \eqref{def:tilde-S} holds, we now want to prove that $u(t)$ converges smoothly 
on every compact subset $V$ of $\cu \setminus \tilde S$ as $t\upto T$. 

Given such a compact subset $V$ of $\cu \setminus \tilde S$ 
we may choose  $r_0\in (0,1)$ small enough that 
\beq \label{est:small_energy} E(u(t),g(t),B_{g(t_0)}(p,r_0))<\eps_0 \quad\text{ for all } t\in [0,T),\text{ and all } p\in V.\eeq

Then choosing  $t_0\in[0,T)$ so that \eqref{ass:compact} holds true for $\Om=V$ and reducing $r_0$ if necessary to ensure that $r_0<\de_{\bar K}(t_0)$ we know from \eqref{est:inj-Om} that we can apply 
both Lemmas \ref{lemma:horiz2} and \ref{lemma:cut-energy} on 
balls $B_{g(t_0)}(p,r)$, $r\leq r_0$, $p\in V$, as they are contained in $\de_{\bar K}(t_0)\thick (M,g(t_0))$.

We first note that \eqref{est:equiv-metric} from Lemma \ref{lemma:horiz2} guarantees that for every $t\in [t_0,T)$
\beq \label{est:inclu} 
B_{g(t_0)}(p,\tfrac{r_0}{C_1})\subset 
B_{g(t)}(p,\tfrac{r_0}{\sqrt{C_1}})
\subset B_{g(t_0)}(p,r_0).\eeq 
We furthermore note that $\tfrac{r_0}{\sqrt{C_1}}$ cannot be larger than $\inj_{g(t)}(p)$ for any $t\in [t_0,T)$ as 
otherwise $B_{g(t)}(p,\tfrac{r_0}{\sqrt{C_1}})$, and thus also $ B_{g(t_0)}(p,r_0)$, would need to contain a curve $\si$ starting and ending in $p$ that is not contractible in $M$, which would contradict the fact that $r_0<\inj_{g(t_0)}(p)$.

Hence $B_{g(t)}(p,\tfrac{r_0}{\sqrt{C_1}})$ is isometric to a ball in the hyperbolic plane and so the smallness of the energy  $E(u(t),g(t),B_{g(t)}(p,\tfrac{r_0}{\sqrt{C_1}}))<\eps_0$ 
obtained from \eqref{est:small_energy} and \eqref{est:inclu} allows us to apply Proposition \ref{prop:eps-reg}
for any $\vph\in  C_c^\infty(B_{g(t_0)}(p,\tfrac{r_0}{C_1}),[0,1])$
and any time $t\in [t_0,T)$. This will be crucial in the proof of the following:

\textbf{Claim:}  For any $p\in V$ and $\vph\in C_c^\infty(B_{g(t_0)}(p,\tfrac{r_0}{C_1}),[0,1])$ (with $r_0>0$ chosen as above) we have 
\beq \label{est:goal} \sup_{t\in [t_0,T)}\int \vph^2\abs{\pt u}^2 dv_g<\infty.
\eeq
In particular there exists a 
 neighbourhood $W$ of $V$
 %, which can be chosen so that  $W\subset \de_{\bar K}(t_0)\thick (M,g(t_0))$, 
so that
  $$\sup_{t\in [t_0,T)}\norm{ u(t)}_{H^2 (W,g(t))}<\infty.$$ 

\textit{Proof of claim:}
To prove the first part of the claim, we differentiate \eqref{eq:evol-u} in time, test with $\vph^2\pt u$ and use that $\ddt dv_g=0$ to write
 \beqa 
 &\tfrac12\tfrac{d}{dt} \int  \vph^2\abs{\pt u}^2\,dv_g+\int \vph^2\abs{d\pt u}^2 \, dv_g\\
 &\qquad =-\int\lan d\pt u,d(\vph^2)\ran_g \cdot \pt u \,dv_g+\tfrac{d}{d\eps}\vert_{\eps=0}\int\Delta_{g(t+\eps)}u\cdot \vph^2 \pt u  \,dv_{g(t+\eps)}\\
 &\qquad \qquad+\int\pt(A_g(u)(du,du))\cdot \vph^2 \pt u\,dv_g\\
 &\qquad\leq \tfrac18\int \vph^2\abs{d\pt u}^2 \, dv_g+C\norm{d\vph}_{L^\infty(M,g)}^2\cdot 
 \norm{\pt u}_{L^2(M,g) }^2
  -\tfrac{d}{d\eps}\vert_{\eps=0}\int \lan du,d(\vph^2\pt u)\ran_{g(t+\eps)} dv_g\\
  &\qquad\qquad +C\norm{\pt g}_{L^\infty(\supp(\vph),g)}^2 E(u,g) +C\int\abs{\pt u}^2 \abs{du}^2_g\vph^2 dv_g\\
 &\qquad\leq \tfrac14 \int  \vph^2\abs{d\pt u}^2dv_g
 + C\norm{d\vph}_{L^\infty(M,g)}^2\norm{\pt u}_{L^2(M,g) }^2+C\norm{\pt g}_{L^\infty(\supp(\vph),g)}^2\\
 &\qquad \qquad  +\hat C\int \abs{\pt u}^2 \abs{du}_g^2\vph^2 dv_g, \label{est:L2-time-der-1}
 \eeqa
  where $C$ and $\hat C$ depend only on a bound $E_0$ on the initial energy and  the target manifold, and the value of $\hat C$ is fixed in what follows.

To estimate the last term in \eqref{est:L2-time-der-1}
 we first apply Proposition \ref{prop:eps-reg} to get 
$$\int\vph^2\abs{\pt u}^2\abs{du}_g^2 dv_g
\leq C
(\int\vph^2\abs{\pt u}^4 dv_g)^{1/2}\cdot \big[ \int\vph^2\abs{\pt u}^2 dv_g+C\norm{d\vph}_{L^\infty(M,g)}^2 \big]^{1/2}.$$
We then recall that $\supp(\vph)$ is contained in the ball $B_{g(t)}(p,\frac{r_0}{\sqrt{C_1}})$ for every $t\in[t_0,T)$ and that $\frac{r_0}{\sqrt{C_1}}\leq \min(\inj_{g(t)}(p),1)$. We may thus view $(\supp(\vph),g(t))$ as a subset of the unit ball in the hyperbolic plane and apply the Sobolev embedding theorem  
to estimate the first factor in the above inequality by
\beqa 
(\int\vph^2\abs{\pt u}^4 dv_g)^{1/2}
&=\norm{\vph \abs{\pt u}^2}_{L^2} %(B_{g(t)}(x_i,\sqrt{C_1}r_i),g(t))}
\leq C\norm{d(\vph  \abs{\pt u}^2)}_{L^1} \\
%(B_{g(t)}(x_i,\sqrt{C_1}r_i),g(t))}\\
&\leq C\norm{\pt u}_{L^2(M,g)}\big(\int\abs{d\pt u}^2\vph^2 dv_g\big)^{1/2}
+C\norm{d\vph}_{L^\infty(M,g)}\norm{\pt u}_{L^2(M,g)}^2.
\eeqa 
Combined this allows us to estimate the final term in \eqref{est:L2-time-der-1} by
$$\hat C\int \abs{\pt u}^2 \abs{du}^2 \vph^2 dv_g\leq \frac14 \int \abs{d\pt u}^2 \vph^2 dv_g+C\norm{\pt u}_{L^2(M,g)}^2
\big[\int \vph^2\abs{\pt u}^2 dv_g+C\norm{d\vph}_{L^\infty(M,g)}^2\big]$$
and thus to reduce \eqref{est:L2-time-der-1} to  
 \beqa 
\frac{d}{dt}\int  \vph^2\abs{\pt u}^2dv_g +\int\vph^2\abs{d\pt u}^2 dv_g&\leq 
C\norm{\pt u}_{L^2(M,g) }^2\cdot \big[ \norm{d\vph}_{L^\infty(M,g)}^2+\int \vph^2 \abs{\pt u}^2dv_g\big] \\
& \quad+C \norm{\pt g}_{L^\infty(\supp(\vph),g)}^2.
\eeqa

Since $\pt g$ is controlled on $\supp(\vph)\subset \de_{\bar K}(t_0)\thick(M,g(t_0))$ by the estimate \eqref{est:bamberg1} of Lemma \ref{lemma:horiz2} while estimate \eqref{est:equiv-metric} from the same lemma implies that
$\norm{d\vph}_{L^\infty(M,g(t))}\leq \sqrt{C_1} \norm{d\vph}_{L^\infty(M,g(t_0))}$, we thus conclude that 
\beqa 
\frac{d}{dt}\int  \vph^2\abs{\pt u}^2dv_g
\leq &\,  C\norm{d\vph}_{L^\infty(M,g(t_0))}^2\norm{\pt  u}_{L^2(M,g) }^2+C \norm{\pt u}_{L^2(M,g)}^2 \int \vph^2 \abs{\pt u}^2dv_g\\
&\quad +C\de_{\bar K}(t_0)^{-1} \norm{\pt g}_{L^2(M,g)}^2
\\
\leq & \,C\big(-\tfrac{dE}{dt}\big)\int \vph^2 \abs{\pt u}^2 dv_g+C\big(-\tfrac{dE}{dt}\big) \cdot \big[\norm{d\vph}_{L^\infty(M,g(t_0))}^2+\de_{\bar K}(t_0)^{-1}\big],
\eeqa
by \eqref{energy-identity},
where $C$ now depends also on the genus of $M$ and $\eta$.
Hence \eqref{est:goal} follows using
Gronwall's lemma. The second part of the claim is now an immediate consequence of \eqref{est:goal} and Proposition \ref{prop:eps-reg}.
%\hfill // 

Based on the claim we have just proven, we can now establish convergence of $u(t)$ in $C^l(V)$ for every $l\in\N$  by well-known arguments: First of all, we may reduce the neighbourhood
$W$ of $V$ if necessary to ensure that $W\subset \de_{\bar K}(t_0)\thick (M,g(t_0))$, compare \eqref{ass:compact} and \eqref{est:inj-Om}. We then apply the Sobolev embedding theorem to obtain that
 $$\sup_{t\in [t_0,T)}\norm{ du(t)}_{L^p (W,g(t))}<\infty \text{ for every } 1\leq p<\infty.$$
The control on the metrics $g(t)$, $t\in [t_0,T)$, obtained in Lemma \ref{lemma:horiz2} thus allows us  to view \eqref{eq:evol-u} as a uniformly parabolic equation on the fixed surface $(W,g(t_0))$ (for times $t$ in this interval $[t_0,T)$)
  whose right-hand side is 
 in $L^p$ for every $p<\infty$. Standard parabolic theory combined with the fact that $u$ is by assumption smooth away from $T$, implies that $u$ is in the parabolic Sobolev space $W^{2,1;p}(\tilde W\times [t_0,T))$ 
 for every $p<\infty$ for a slightly smaller neighbourhood $\tilde W$ of $V$. 
 In particular $u$ is H\"older continuous with exponent $\alpha$ for every $\alpha<1$ on $\tilde W \times [0,T)$.
 
Taking covariant derivatives $\na_{g(t)}^l$ of \eqref{eq:evol-u} allows us to repeat the above argument and obtain that $(x,t)\mapsto (\na_{g(t)}^l u)(x,t)$ is H\"older continuous on $V\times [t_0,T)$ for every $l\in \N$. As the metrics converge smoothly to $h$ on $V$, this allows us to conclude that also
 $u(t)\to \bar u$ in $C^l(V,h)$ 
 for every $l\in\N$, for some
 $\bar u$.

As the obtained convergence implies in particular that the set $S$ defined in \eqref{def:S}, as used in Proposition \ref{prop:conv-nondeg}, agrees with the set $\tilde S$ of points satisfying \eqref{def:tilde-S} considered here, this completes the proof of  part \ref{part-1-prop} of Proposition \ref{prop:conv-nondeg}.

For the proof of part \ref{part-2-prop} of the proposition we closely follow the arguments of \cite[Section 2]{finwind}.

Let $p\in S$. As above we choose $t_0<T$ so that \eqref{ass:compact}
 holds true for $\Om=\{p\}$ which we recall allows us to 
 apply Lemmas \ref{lemma:horiz2} and \ref{lemma:cut-energy} on balls $B_{g(t_0)}(p,r_0)$, 
 $r_0\in (0,\de_{\bar K}(t_0))$ since \eqref{est:inj-Om} ensures that such balls are contained in $\de_{\bar K}(t_0)\thick(M,g(t_0))$. We fix such a radius $r_0$ which is small enough so that $B_{g(t_0)}(p,r_0)$ contains no other element of the singular set $S$. 

Given any fixed cut-off function $\psi\in C_c^\infty([0,1),[0,1])$
with $\psi\equiv 1$ on $[0, \frac12]$ and with $\norm{\psi'}_{L^\infty}\leq 4$
we set
$$\vph_{r}(x):=\psi\big(\tfrac{\dist_{g(t_0)}(p,x)^2}{r^2}\big), \quad 0<r<r_0$$
and note that $\norm{d\vph_r}_{L^\infty(M,g(t_0))}\leq \frac{C}{r}$. As  $\supp(\vph_r)\subset B_{g(t_0)}(p,r_0)$ we can apply  Lemma \ref{lemma:cut-energy} to control the associated
cut-off energies 
$E_r(t):=E_{\vph_r}(t)$ defined in \eqref{def:cut-energy} and obtain in particular
that  
$\lim_{t\upto T} E_r(t)$ exists for every $r\in (0,r_0)$. 
Combined with 
the local $C^l$ convergence of $u(t)\to \bar u$ on  $\cu\setminus S$ and the convergence of the metrics obtained in part \ref{1.1:1} of Theorem \ref{thm:basic_convergence} this 
implies that  
\beq \label{est:Ephat}\hat E_p:=\lim_{t\upto T}  E_r(t)-\half \int\vph_r^2\abs{d\bar u}_h^2 dv_h\eeq
is independent of $r\in (0,r_0)$.

Let now  $\nu>0$.  For $t\in[t_0,T)$ sufficiently  close to $T$ so that $\nu(T-t)^\half<r_0$ we can apply Lemma \ref{lemma:cut-energy} to $s\mapsto E_{\nu(T-t)^{1/2}}(s), s\in [t_0,T)$, 
in order to obtain the second inequality of  
\beqa  \label{est:E_nu}
\big|E_{\nu (T-t)^{1/2}}(t)-\hat E_p\big| & \leq
\big|E_{\nu (T-t)^{1/2}}(t)-\lim_{s\upto T}E_{\nu (T-t)^{1/2}}(s)\big|+
\half \int \vph_{\nu(T-t)^{1/2}}^2\abs{d\bar u}_h^2 dv_h 
\\
 &\leq E(t)-E(T)+C[\nu^{-1}+\de_{\bar K}^{-\half}(t_0)\cdot (T-t)^\half]\cdot(E(t)-E(T))^\half \\
&\qquad + E(\bar u, h, B_{g(t_0)}(p,\nu(T-t)^{1/2})).
\eeqa
We furthermore note that $B_{g(t_0)}(p,\nu(T-t)^{1/2}))\subset B_{h}(p,\sqrt{C_1}\nu(T-t)^{1/2})$, compare \eqref{est:equiv-metric} of Lemma \ref{lemma:horiz2}, and thus that the last term in \eqref{est:E_nu} tends to zero as $t\upto T$.
Passing to the limit $t\upto T$ in \eqref{est:E_nu} we thus obtain that 
 also 
$$\lim_{t\upto T}E_{\nu (T-t)^{1/2}}(t)=\hat E_p\text{ for every }\nu>0.$$
Combined with the equivalence \eqref{est:equiv-metric} of the metrics obtained in Lemma \ref{lemma:horiz2} we therefore get that for any $\nu>0$
\beqa 
\lim_{r\downarrow 0}\lim_{t\upto T} E(u(t),g(t), B_{g(t)}(p,r))\leq &\,
\lim_{r\downarrow 0}\lim_{t\upto T} E(u(t),g(t), B_{g(t_0)}(p,\sqrt{C_1}r))\leq
\lim_{r\downarrow 0}\lim_{t\upto T} E_{2\sqrt{C_1}r}(t)
\\
=&\,
\label{est:scale-bubbles} 
\hat E_p=
\lim_{t\upto T}E_{\nu C_1^{-1/2} (T-t)^{1/2}}(t)\\
\leq &\,\liminf_{t\upto T} E(u(t),g(t), B_{g(t_0)}(p,\nu C_1^{-1/2}(T-t)^{1/2}))\\
\leq &\,\liminf_{t\upto T} E(u(t),g(t), B_{g(t)}(p,\nu (T-t)^{1/2})).
\eeqa
As the `reverse' inequality 
$$\limsup_{t\upto T} E(u(t),g(t), B_{g(t)}(p,\nu (T-t)^{1/2}))\leq
\lim_{r\downarrow 0}\lim_{t\upto T} E(u(t),g(t), B_{g(t)}(p,r))$$
is trivially true, this proves 
the second equality in \eqref{eq:energy-id-balls}, 
including the existence of the limits taken,
while the first inequality of \eqref{eq:energy-id-balls} follows directly from the equivalence \eqref{est:equiv-metric} of the metrics $g(t)$ and $h$ obtained in Lemma \ref{lemma:horiz2}.

To establish the final inequality of \eqref{eq:energy-id-balls} we closely follow \cite[Section 2]{finwind}.
Given a sequence of times $t_n\upto T$ as in \eqref{tau_control}
and a point $p\in S$ we pick local isothermal coordinates centred at $p$ for each of the $g(t_n)$ by identifying $B_{g(t_n)}(p,r_0)$ with the corresponding ball centred at zero of the Poincar\'e hyperbolic disc, viewed conformally as the unit disc centred at the origin in $\R^2$,
and 
rescale to obtain a sequence of maps 
$$u_n(x):=u(r_n x,t_n), \quad r_n:=(T-t_n)^{1/2} $$ 
for which $\norm{\tau(u_n)}_{L^2(\ck)}\to 0$ for every $\ck\subset\subset \R^2$.

Since \eqref{est:scale-bubbles} implies that 
$E(u_n, B(0,\Lambda)\setminus B(0,\lambda))\to 0$ for any 
$0<\la<\La$, a subsequence of 
the maps $u_n$ converges strongly in $H^1$ away from $0$ to a constant map while 
 bubbles $\{\om_j\}_{j=1}^{m'}$ develop near the origin at scales $\hat \la_n^j\to 0, n\to \infty$.

The scales at which the bubbles $\om_j$ develop in the original sequence are thus $\la_n^j=r_n\hat \la_n^j=o((T-t_n)^{1/2})$ and the `no-loss-of-energy' result for bubble tree convergence of almost harmonic maps of \cite{DT}
ensures that all the energy of the $u_n$ is captured by these bubbles i.e. that for every $\La>0$ we have
$$\lim_{n\to\infty} E(u_n,B_\La(0))=\sum_{l=1}^{m'} E(\om_l).$$
Taking the limit $\La\downto 0$, and bearing in mind that 
all but the final equality of \eqref{eq:energy-id-balls} has already been established, we find that 
for every $p\in S$, we have
\beqa 
\lim_{r\downarrow 0}\lim_{t\upto T} E(u(t),g(t), B_{g(t)}(p,r))
%=\lim_{t\upto T}E(u(t),g(t),B_{g(t)}(p,\nu(T-t)^\half))
=  \sum_{l=1}^{m'} E(\om_l),
\eeqa
completing the proof of \eqref{eq:energy-id-balls}.

Finally, given any compact subset $\Om\subset\subset \cu$ which is large enough for $S$ to be contained in the interior of $\Om$ 
we can combine \eqref{eq:energy-id-balls} with the strong 
 $H^1_{loc}$ convergence of $u(t)\to \bar u$ on $\cu\setminus S$ and the convergence of the metrics to obtain that indeed
\beq
\label{partial_limit}
\lim_{t\upto T} E(u(t),g(t),\Om)= E(\bar u, h, \Om) + \sum_{l=1}^{m''} E(\om_l)
\eeq
where $\{\om_l\}_{l=1}^{m''}$ is the set of all bubbles developing at points in $S$ along a sequence of times $t_n$ as considered in the proposition. 
\end{proof}

\subsection{All energy lost down collars is represented by bubbles}

\label{sect:deg}

At this point we have a good description of the convergence of $u(t)$ and $g(t)$ locally on $\cu=M\setminus  F$, with Proposition \ref{prop:conv-nondeg} completing the proof of Theorem \ref{thm:basic_convergence} and establishing part \ref{part1} of Theorem \ref{thm:bubbles}.
In this section we prove parts \ref{part:energy_limit} and \ref{part_neck} of 
Theorem \ref{thm:bubbles}, which show that near the centre of degenerating collars, the map is looking like a collection of bubbles, while on larger scales that are nevertheless vanishing scales, 
where we have no way of showing that the map is becoming harmonic, no energy can be lost.

\begin{proof}[Proof of part \ref{part:energy_limit} of Theorem \ref{thm:bubbles}]
As a next step we now prove part \ref{part:energy_limit} of Theorem \ref{thm:bubbles}
which can be seen as quantifying the size of the part of 
%$M\setminus  F$ 
$\cu$
on which the energy has almost reached its limit. As we can only apply the local energy estimate from Lemma \ref{lemma:cut-energy} on regions with sufficiently large injectivity radius, we will obtain the existence of a limit of the energy on the $[T-t]\thin$ part by proving that the limit on the $[T-t]\thick$ part exists and 
agrees with $E_{thick}$ and then appealing to the existence of a limit of the total energy $E(t)$.

As above
it will be more convenient to work not with energies over given sets, but with cut-off energies $E_\vph$ as defined in \eqref{def:cut-energy}. 
To this end we let $\de_K(t)=K(T-t)(E(t)-E(T))$, $K\geq \bar K$, be as in Lemma \ref{lemma:horiz2} and recall that the characterisation of the pinching set \eqref{eq:char-pinching} implies in particular that for every $t_0\in[0,T)$ 
$$\inj_{g(t_0)}(M)< \de_K(t_0)$$
and thus that 
$$A_{K,t_0}:=\{x\in M\ :\ \inj_{g(t_0)}(x)\leq\de_{K}(t_0)\}$$
is nonempty.
% for every $K\geq \bar K$ and every $t_0\in [0,T)$ for which 
% $\de_{K}(t_0)\leq \max_{x\in M}\inj_{g(t_0)}(x)$, and is thus in particular nonempty for $t_0\in [0,T)$ sufficiently large, depending in particular on $K$. 
%In what follows, we will always assume that $t_0$ is this large.
%
We will always assume that $t_0\in [0,T)$ is sufficiently large, depending in particular on $K$, so that $\de_K(t_0)\cdot (\pi e)<\arsinh(1)$. In this way, not only can we be sure that every point in $A_{K,t_0}$ has injectivity radius less than $\arsinh(1)$, and is thus lying within some collar region around a geodesic of length less than $2\arsinh(1)$, 
we can also be sure that the $1$-fattening of 
$A_{K,t_0}$, i.e. $\{p\in M\ |\ \dist_{g(t_0)}(p,A_{K,t_0})<1\}$, must 
lie within $\de_{e\pi K}(t_0)\thin(M,g(t_0))$, and hence 
also lie within a union of such (pairwise disjoint) collars, since by
\cite[Lemma A.3]{RT-horizontal}  if $x\in A_{K,t_0}$ and $y\in B_{g(t_0)}(x,1)$ lies in the same collar, then $\inj_{g(t_0)}(y)\leq \inj_{g(t_0)}(x) \cdot (\pi e)
\leq \de_K(t_0)\cdot (\pi e)<\arsinh(1)$, so we cannot escape this collar within a distance $1$ of $x$. In particular, the function $x\mapsto \dist_{g(t_0)}(x,A_{K,t_0})$ is smooth on the 
$1$-fattening of $A_{K,t_0}$.

%\bcmt{Important to digest here that $y$ is not the point in the $1$-fattening that we are trying to show lies in the collar, since we are assuming that it lies in the collar! This is a continuity argument}

Given any 
smooth cut-off function $\phi:\R\to [0,1]$ such that $\phi(x)=0$ for $x\leq 0$, $\phi(x)=1$ for $x\geq 1$ and $|\phi'|\leq 2$,
we can thus define the induced \emph{smooth} cut-off $\vph_{K,t_0}:M\to [0,1]$ by 
\beq 
\vph_{K,t_0}(x):=
\phi(\dist_{g(t_0)}(x,A_{K,t_0})).
\eeq
It is immediately apparent that
\beq
\label{vph0}
\vph_{K,t_0}\equiv 0\qquad\text{on }\de_K(t_0)\thin(M,g(t_0)),
\eeq
and  that the support of $\vph_{K,t_0}$ lies within $\de_K(t_0)\thick(M,g(t_0))$ and hence
$\vph_{K,t_0}$ has compact support within $\cu$ owing to \eqref{Uchar}.
This will shortly allow us to apply Lemma \ref{lemma:cut-energy} to the 
corresponding local energy 
$E_{K,t_0}(t):=E_{\vph_{K,t_0}}(t)$ that serves as a substitute for the energy of $u(t)$ over $\de_K(t_0)\thick(M,g(t_0))$.

We also claim that
\beq
\label{vph1}
\vph_{K,t_0}\equiv 1\qquad\text{on }\de_{e\pi K}(t_0)\thick(M,g(t_0)).
\eeq
Indeed, the only way this could fail would be if we could find a point in the $1$-fattening of
$A_{K,t_0}$ that lies in $\de_{e\pi K}(t_0)\thick(M,g(t_0))$, which we ruled out above.

By \eqref{vph0}, we see that 
$E_{K,t_0}(t)\leq E\big(u(t),g(t),\de_K(t_0)\thick (M,g(t_0))\big)$, and so 
\beq \label{est:triv_dir_EKt}
\lim_{K\to \infty}\limsup_{t\upto T}E_{K,t}(t)\leq \lim_{K\to \infty}\limsup_{t\upto T} E(u(t),g(t),\de_K(t)\thick(M,g(t))).
\eeq
On the other hand, 
by \eqref{vph1}, we see that
$E(u(t),g(t),\de_{e\pi K}(t_0)\thick(M,g(t_0)))\leq E_{K,t_0}(t)$, 
and hence we have the converse inequality
\beq 
\label{est:nontriv_dir_EKT}
\lim_{K\to \infty}\limsup_{t\upto T} E(u(t),g(t),\de_{K}(t)\thick (M,g(t))\leq \lim_{K\to \infty}\limsup_{t\upto T} E_{K,t}(t),
\eeq
i.e. we have equality in \eqref{est:triv_dir_EKt} and \eqref{est:nontriv_dir_EKT}.
Therefore to prove \eqref{en_id_thin2}, 
it suffices to show that 
\beq \label{old-destinationE}
E_{thick}=\lim_{K\to \infty} \limsup_{t\upto T}E_{K,t}(t).\eeq
We claim first that
\beq
\label{Ethickalt}
E_{thick}=\limsup_{t_0\upto T}\lim_{t\upto T} E_{K, t_0}(t)
\eeq
where the existence of $\lim_{t\upto T} E_{K, t_0}(t)$ is guaranteed by Lemma \ref{lemma:cut-energy}.
To see \eqref{Ethickalt}, first recall that for $K$, $t_0$ as above, the support of 
$\vph_{K,t_0}$ is compact within $\cu$, and is thus contained within 
$\de\thick(\cu,h)$ for sufficiently small $\de>0$. 
By reducing $\de$ further, we may assume that all bubble points in $S$ lie within the interior of $\de\thick(\cu,h)$.
Therefore we have
$E(u(t),g(t),\de\thick(\cu,h))\geq E_{K,t_0}(t)$, and taking the limits $t\upto T$,
$\de\downto 0$ and $t_0\upto T$ in that order, we find that
$E_{thick}\geq\limsup_{t_0\upto T}\lim_{t\upto T} E_{K, t_0}(t)$.
To see the converse inequality, we observe that by \eqref{vph1}, for any $\de>0$ and $t_0<T$ sufficiently large (depending on $\de$, $K$ etc.) we have 
$\vph_{K,t_0}\equiv 1$ on $\de\thick(M,g(t_0))$, and so 
$E(u(t),g(t),\de\thick(\cu,h))\leq E_{K,t_0}(t)$. This time we take limits in the order
$t\upto T$, $t_0\upto T$ and then $\de\downto 0$ to give
$E_{thick}\leq\limsup_{t_0\upto T}\lim_{t\upto T} E_{K, t_0}(t)$, and hence
\eqref{Ethickalt}.

Thus 
\eqref{en_id_thin2}
would follow if we can prove that
as $K\to \infty$ we have
\beq \label{new-destinationE}
\limsup_{t_0\upto T}\abs{ E_{K,t_0}(t_0)-\lim_{t\upto T}E_{K,t_0}(t)}\to 0.\eeq
But this follows from Lemma \ref{lemma:cut-energy}, which implies  that
for  $t_0\in [0,T)$ as large as considered above, and every $t\in [t_0,T)$, we have
\beq \abs{E_{K,t_0}(t)-E_{K,t_0}(t_0)}\leq  E(t_0)-E(T)+
\frac{C}{K^\half}+C(T-t_0)^{\frac12}(E(t_0)-E(T))^{\half}\label{est:EK-diff}\eeq
with $C$ depending only on the genus of $M$, $\eta$ and an upper bound on the initial energy, which thus yields \eqref{new-destinationE} after taking the limits $t\upto T$, $t_0\upto T$
and $K\to\infty$, in that order. 

Now that \eqref{en_id_thin2} has been proved, we verify that \eqref{en_id_thin1} follows as a result. In particular, we verify that the limit taken in \eqref{en_id_thin1} exists.
However large we take $K>0$, for sufficiently large $t<T$ we have
$T-t\geq \de_K(t)$, and hence
$$E(u(t),g(t),[T-t]\thin(M,g(t)))\geq E(u(t),g(t), \de_K(t)\thin(M,g(t))).$$
Taking a $\liminf$ as $t\upto T$ and then the limit $K\to\infty$, and using
\eqref{en_id_thin2} we find that
\beq
\label{onewayround}
\liminf_{t\upto T}E(u(t),g(t),[T-t]\thin(M,g(t)))\geq E_{thin}.
\eeq
To obtain the converse inequality, observe that given any $\de>0$, for sufficiently large
$t<T$ we have
$\de\thin(\cu,h)\supset [T-t]\thin(M,g(t))$, cf. \eqref{conv:inj}, and therefore
$$E(u(t),g(t),\de\thin(\cu,h))\geq E(u(t),g(t),[T-t]\thin(M,g(t))).$$
Provided $\de>0$ is sufficiently small (so that the singular set $S$ is in the interior of $\de\thick(\cu,h)$), we can then take a limit as $t\upto T$,
followed by a limit as $\de\downto 0$, to give
$$E_{thin}\geq \limsup_{t\upto T}E(u(t),g(t),[T-t]\thin(M,g(t))),$$
which when combined with \eqref{onewayround} completes the proof of \eqref{en_id_thin1}
and hence of part \ref{part:energy_limit} of the theorem.
\end{proof}

While part \ref{part:energy_limit} of Theorem \ref{thm:bubbles} gives good control on where energy can concentrate on the degenerating part of the surface, we currently have no control of what parts of the map are lost down the degenerating parts of the collar at the singular time $T$.
This is addressed by part \ref{part_neck}, which we shall now prove.

\begin{proof}[Proof of part \ref{part_neck} of Theorem \ref{thm:bubbles}]
Proposition \ref{prop:collar_desc} tells us that the length $\ell(t_n)$ of the central geodesic of each degenerating collar is controlled like $\ell(t_n)=o(T-t_n)$ and hence that the $[T-t_n]\thin$ part of such a collar, where all of the lost energy lives, is represented by 
longer and longer cylinders 
$\tilde\Col_n:= \Col(t_n, \de_n)=
(-\tilde X_n,\tilde X_n)\times S^1$, $\de_n=T-t_n$,  equipped with the  corresponding collar metrics $g=\rho^2g_0$.

We can indeed consider the maps on the 
larger subcollars $\widehat\Col_n=(-\widehat X_n,\widehat X_n)$  which correspond to the $[T-t_n]^{\half}\thin$ parts of the collar, where we note that  $1\ll \tilde X_n \ll \widehat X_n \ll X(\ell_n)$, compare \eqref{eq:Xj}. 

We recall from \cite[(A.9)]{RT3} 
that $\rho(y)\leq \inj_{g(t)}(y)$ as $y$ varies within each collar. Therefore, throughout $\widehat\Col_n$ we have
$\rho\leq (T-t_n)^\half$. 
By the scaling of the tension field, 
%WRONG:(see e.g. \cite[(A.4)]{HRT}) 
if we switch from the hyperbolic metric $g_n=g(t_n)$ to the flat cylinder metric 
$g_0=ds^2+d\th^2$ on each such subcollar, then we can estimate the tension of $u_n:=u(t_n)$ according to
\beq
\label{est:almost-harmonic}
\|\tau_{g_0}(u_n)\|_{L^2(\widehat\Col_n,g_0)}\leq 
(\sup_{\widehat\Col_n}\rho)
\|\tau_{g_n}(u_n)\|_{L^2(\widehat\Col_n,g_n)}
\leq
(T-t_n)^\half \|\tau_{g_n}(u_n)\|_{L^2(M,g_n)}
\to 0
\eeq
by \eqref{tau_control}.

We can thus view the $u_n$'s as almost-harmonic maps from longer and longer cylinders $(\widehat \Col_n,g_0)$ and apply  
Proposition \ref{prop:bb-convergence}
to pass to a subsequence that converges to a full bubble branch.

It is this estimate \eqref{est:almost-harmonic} and the precise information on the degenerate region where energy can concentrate obtained in part \ref{part:energy_limit}  of Theorem \ref{thm:bubbles}
that allows us to represent the maps on these parts in terms of branched minimal immersions and curves.  
We stress that we would not be able to perform this analysis on the whole collar.

We also remark that in our situation we obtain the additional information that 
 any bubble obtained in the convergence to a full bubble branch described in Proposition \ref{prop:bb-convergence} will be contained 
in the $(T-t_n)\thin$ part of the surface as we already know that no energy can be lost on $\{p: \inj_{g(t)}(p)\in [(T-t),(T-t)^{1/2}]$\}.
\end{proof}

{\sc MR: 
Mathematical Institute, University of Oxford, Oxford, OX2 6GG, UK}\\
{\sc PT: Mathematics Institute, University of Warwick, Coventry,
CV4 7AL, UK}

\end{document}